\documentclass[a4paper,10pt]{amsart}

\usepackage[french, english]{babel} 
\usepackage[T1]{fontenc}
\usepackage[ansinew]{inputenc}

\date{}
\title[Fractional Ornstein-Uhlenbeck operators]{Smoothing properties of fractional Ornstein-Uhlenbeck semigroups and null-controllability}

\author{Paul Alphonse}
\address{Paul Alphonse, Univ Rennes, CNRS, IRMAR - UMR 6625, F-35000 Rennes}
\email{paul.alphonse@ens-rennes.fr}

\author{Joackim Bernier}
\address{Joackim Bernier, Univ Rennes, CNRS, IRMAR - UMR 6625, F-35000 Rennes}
\email{joackim.bernier@ens-rennes.fr}

\keywords{Fractional Ornstein-Uhlenbeck operators, Kalman rank condition, Gevrey regularity, Null-controllability, Subelliptic estimates}

\subjclass[2010]{93B05, 35B65, 35H20}

\usepackage[top=3cm, bottom=2cm, left=3cm, right=3cm]{geometry}	

\usepackage{enumitem}

\usepackage{array}											

\usepackage{amsfonts}
\usepackage{amsmath}
\usepackage{amsthm}
\usepackage{amssymb}

\usepackage{bbm}

\usepackage{stmaryrd}

\usepackage[scr]{rsfso}

\usepackage{comment}

\usepackage{hyperref}

\usepackage{xcolor}

\usepackage{constants}
\newconstantfamily{c}{symbol=c}

\numberwithin{equation}{section}
										
\newtheorem{thm}{Theorem}[section]
\newtheorem{prop}[thm]{Proposition}
\newtheorem{lem}[thm]{Lemma}
\newtheorem{cor}[thm]{Corollary}
\theoremstyle{definition}
\newtheorem{dfn}[thm]{Definition}
\newtheorem{ex}[thm]{Example}
\newtheorem{rk}[thm]{Remark}

\DeclareMathOperator{\Supp}{Supp}
\DeclareMathOperator{\Reelle}{Re}

\DeclareMathOperator{\Tr}{Tr}
\DeclareMathOperator{\Ker}{Ker}
\DeclareMathOperator{\Ran}{Ran}

\DeclareMathOperator{\Rank}{Rank}

\DeclareMathOperator{\Lie}{Lie}

\begin{document}

\sloppy

\selectlanguage{english}

\begin{abstract} We study fractional hypoelliptic Ornstein-Uhlenbeck operators acting on $L^2(\mathbb{R}^n)$ satisfying the Kalman rank condition. We prove that the semigroups generated by these operators enjoy Gevrey regularizing effects. Two byproducts are derived from this smoothing property. On the one hand, we prove the null-controllability in any positive time from thick control subsets of the associated parabolic equations posed on the whole space. On the other hand, by using interpolation theory, we get global $L^2$ subelliptic estimates for the these operators.
\end{abstract}

\maketitle

\section{Introduction}
\label{lab_intro}

\subsection{Motivation} Given $s>0$ a positive real number, $B = (B_{i,j})_{1\le i,j\le n}$ and $Q=(Q_{i,j})_{1\le i,j\le n}$ real $n\times n$ matrices, with $Q$ symmetric positive semidefinite, we aim in this work at studying the fractional Ornstein-Uhlenbeck operator
\begin{gather}\label{06022018E3}
	\mathcal{P} = \frac{1}{2}\Tr^s(-Q\nabla^2_x) + \langle Bx,\nabla_x\rangle,\quad x\in\mathbb{R}^n,
\end{gather}
equipped with the domain
\begin{equation}\label{19062018E1}
	D(\mathcal{P}) = \big\{u\in L^2(\mathbb{R}^n) : \mathcal{P}u\in L^2(\mathbb{R}^n)\big\}.
\end{equation}
This operator is composed of $\Tr^s(-Q\nabla^2_x)$ the Fourier multiplier whose symbol is $\langle Q\xi,\xi\rangle^s$, where $\langle\cdot,\cdot\rangle$ stands for the canonical Euclidean scalar product on $\mathbb{R}^n$, and $\langle Bx,\nabla_x\rangle$ the differential operator defined by
$$\langle Bx,\nabla_x\rangle = \sum_{i=1}^n\sum_{j=1}^nB_{i,j}x_j\partial_{x_i}.$$
Under an algebraic condition on $B$ and $Q^{\frac{1}{2}}$ (the symmetric positive semidefinite matrix given by the square root of $Q$), we investigate the regularizing effects of the semigroup $(e^{-t\mathcal{P}})_{t\geq0}$ generated by $\mathcal{P}$ on $L^2(\mathbb{R}^n)$, the null-controllability of the parabolic equation associated to $\mathcal{P}$ and the global $L^2$ subelliptic properties enjoyed by $\mathcal{P}$. This algebraic condition is the so-called Kalman rank condition
\begin{align}\label{10052018E4}
	\Rank[B\ \vert\ Q^{\frac{1}{2}}] = n,
\end{align}
where 
$$[B\ \vert\ Q^{\frac{1}{2}}] = [Q^{\frac{1}{2}},BQ^{\frac{1}{2}},\ldots,B^{n-1}Q^{\frac{1}{2}}],$$
is the $n\times n^2$ matrix obtained by writing consecutively the columns of the matrices $B^jQ^{\frac{1}{2}}$. Equivalently, $B$ and $Q^{\frac{1}{2}}$ satisfy the Kalman rank condition when there exists a non-negative integer $0\le r\le n-1$ satisfying
\begin{equation}\label{05052018E4}
	\Ker(Q^{\frac{1}{2}})\cap\Ker(Q^{\frac{1}{2}}B^T)\cap\ldots\cap\Ker(Q^{\frac{1}{2}}(B^T)^r) = \{0\}.
\end{equation}
This equivalence is proved in Lemma \ref{29082018E1} in Appendix.

A particular case of fractional Ornstein-Uhlenbeck operator is the fractional Kolmogorov operator
$$\mathcal{P} = v\cdot\nabla_x + (-\Delta_v)^s,\quad (x,v)\in\mathbb{R}^{2n},$$
where $(-\Delta_v)^s$ is the Fourier multiplier of symbol $|\eta|^{2s}$, with $\eta\in\mathbb R^n$ the dual variable of $v\in\mathbb R^n$, obtained for
\begin{equation}\label{08062018E3}
	B = \begin{pmatrix}
	0_n & I_n \\
	0_n & 0_n
\end{pmatrix}\quad \text{and}\quad Q = 2^{\frac{1}{s}}\begin{pmatrix}
	0_n & 0_n \\
	0_n & I_n
\end{pmatrix}.
\end{equation}
Note that, here, implicitly, to be consistent with the definition of the Ornstein-Uhlenbeck operators in \eqref{06022018E3}, we have realized the change of notations $x\leftarrow (x,v)$ and $n\leftarrow 2n$.

This operator plays a substantial role in kinetic theory since the fractional Kolmogorov equation
$$\left\{\begin{array}{l}
	\partial_t u(t,x,v) + v\cdot\nabla_xu(t,x,v) + (-\Delta_v)^su(t,x,v) = 0 ,\quad t>0,\ (x,v)\in\mathbb{R}^{2n}, \\[5pt]
	u(0) = u_0\in L^2(\mathbb{R}^n),
\end{array}\right.$$
where $0<s<1$, turns out to be a simplified model of the linearized spatially inhomogeneous non-cutoff Boltzmann equation. We refer the reader e.g. to \cite{MR2679369, MR2784329, MR3348825, MR1942465} for extensive discussions about this topic.

The fractional Ornstein-Uhlenbeck operators also appear in stochastic theory. Considering the stochastic differential equation in $\mathbb{R}^n$,
$$\left\{\begin{array}{l}
	dX_t = BX_t\ dt + Q^{\frac{1}{2}}dN_t, \\[5pt]
	X_0 = x\in\mathbb{R}^n,
\end{array}\right.$$
where $N_t$ stands for a $2s$-stable L\'evy process, the fractional Ornstein-Uhlenbeck semigroup is the transition semigroup of the process $(X_t)_{t\geq0}$, see e.g. Examples 1.3.14 and 3.3.8 in \cite{MR2512800}.

In the rest of the introduction, we denote by
\begin{equation}\label{05052018E1}
	\mathcal{L} = -\frac{1}{2}\Tr(Q\nabla^2_x) + \langle Bx,\nabla_x\rangle
\end{equation}
the usual Ornstein-Uhlenbeck operator, corresponding to the case when $s = 1$. These operators acting on Lebesgue spaces have been very much studied in the last two decades. The structure of these operators was analyzed in \cite{MR1289901}, while their spectral properties were investigated in \cite{MR1941990,  MR3342487}. The smoothing properties of the associated semigroups were studied in \cite{MR2505366, MR2257846, MR3710672,MR1389786,MR1941990, MR3342487} and some global hypoelliptic estimates were derived in \cite{MR2729292, MR2257846, MR3710672, MR3342487}. We also refer the reader to \cite{MR1343161, MR1475774} where the operator $\mathcal L$ is studied while acting on spaces of continuous functions. We recall from these works that the hypoellipticity of the Ornstein-Uhlenbeck operator $\mathcal L$ is characterized by the following equivalent assertions : 
\begin{enumerate}[label=\textbf{\arabic*.},leftmargin=* ,parsep=2pt,itemsep=0pt,topsep=2pt]
\item The Ornstein-Uhlenbeck operator $\mathcal L$ is hypoelliptic.
\item The symmetric positive semidefinite matrices 
\begin{equation}\label{27082018E2}
	Q_t = \int_0^te^{-sB}Qe^{-sB^T}\ ds,
\end{equation}
are nonsingular for some (equivalently, for all) $t>0$, i.e. $\det Q_t>0$.
\item The Kalman rank condition \eqref{10052018E4} holds.
\item The H\"{o}rmander condition holds : $$\forall x\in\mathbb{R}^n,\quad \dim\Lie(X_1,X_2,\ldots,X_n,Y_0)(x) = n,$$
with $$Y_0 = \langle Bx,\nabla_x\rangle,\quad X_i = \sum_{j=1}^nq_{i,j}\partial_{x_j},\quad i=1,\ldots,n,$$
where $\Lie(X_1,X_2,\ldots,X_n,Y_0)(x)$ denotes the Lie algebra generated by the vector fields 
$$X_1,X_2,\ldots,X_n\quad \text{and}\quad Y_0,$$
at point $x\in\mathbb{R}^n$. 
\end{enumerate}

\subsection{Regularizing effects of semigroups} First, we derive an explicit formula for the semigroup generated by $\mathcal{P}$ on $L^2(\mathbb{R}^n)$. The case of the hypoelliptic Ornstein-Uhlenbeck operator $\mathcal{L}$ is treated by Kolmogorov in \cite{MR1503147}, where he proves that the semigroup $(e^{-t\mathcal{L}})_{t\geq0}$ generated by $\mathcal{L}$ has the following explicit representation :
$$e^{-t\mathcal{L}}u(x) = \frac{1}{(2\pi)^{\frac{n}{2}}\sqrt{\det Q_t}}\int_{\mathbb{R}^n}e^{-\frac{1}{2}\langle (Q_t)^{-\frac{1}{2}}y,y\rangle}u(e^{-tB}x-y)dy,$$
when $t>0$, where the symmetric positive semidefinite matrices $Q_t$ are defined in \eqref{27082018E2}. Since $e^{-t\mathcal{L}}u$ is given by a convolution, it follows from the properties of the Fourier transform that the above formula writes as
\begin{equation}\label{04052018E3}
	\widehat{e^{-t\mathcal{L}}u}(\xi)= e^{\Tr(B)t}\exp\left[-\frac{1}{2}\int_0^t\vert Q^{\frac{1}{2}}e^{\tau B^T}\xi\vert^2\ d\tau\right]\widehat{u}(e^{tB^T}\xi).
\end{equation}
In this work, without any assumption on $B$ and $Q^{\frac{1}{2}}$, we prove (after studying its basic properties) that the operator $\mathcal P$ generates a strongly continuous semigroup on $L^2(\mathbb{R}^n)$ and we derive an explicit formula for its Fourier transform, extending \eqref{04052018E3} :

\begin{thm}\label{06022018T1} The fractional Ornstein-Uhlenbeck operator $\mathcal{P}$ defined in \eqref{06022018E3} and equipped with the domain \eqref{19062018E1} generates a strongly continuous semigroup $(e^{-t\mathcal{P}})_{t\geq0}$ on $L^2(\mathbb{R}^n)$ which satisfies that for all $t\geq0$ and $u\in L^2(\mathbb{R}^n)$,
$$\Vert e^{-t\mathcal{P}}u\Vert_{L^2(\mathbb{R}^n)}\le e^{\frac{1}{2}\Tr(B)t} \Vert u\Vert_{L^2(\mathbb{R}^n)},$$
and
$$\widehat{e^{-t\mathcal{P}}u} = e^{\Tr(B)t}\exp\left[-\frac{1}{2}\int_0^t\vert Q^{\frac{1}{2}}e^{\tau B^T}\cdot\vert^{2s}\ d\tau\right]\widehat{u}(e^{tB^T}\cdot).$$
\end{thm} 

A natural question is then to investigate the regularizing properties of this semigroup. In this direction, Y. Morimoto and C.J. Xu proved in \cite{MR2523694} that any solution $u$ of the fractional Kolmogorov equation 
$$\left\{\begin{array}{l}
	\partial_tu(t,x,v) + v\cdot(\nabla_xu)(t,x,v) + (-\Delta_v)^su(t,x,v) = 0,\quad t>0,\ (x,v)\in\mathbb{R}^{2n}, \\[5pt]
	u(0) = u_0\in L^2(\mathbb{R}^{2n}),
\end{array}\right.$$
belongs to the Gevrey type space $G^{\frac{1}{2s}}(\mathbb{R}^{2n})$ for any time $t>0$. In all this paper, we denote by $G^{\frac{1}{2s}}(\mathbb{R}^n)$ the space of regular functions $u\in C^{\infty}(\mathbb{R}^n)$ satisfying
\begin{equation}\label{31072018E1}
	\exists c>1,\forall\alpha\in\mathbb{N}^n,\quad \Vert \partial^{\alpha}_xu(x)\Vert_{L^2(\mathbb{R}^n)}\le c^{1+\vert\alpha\vert}\ (\alpha!)^{\frac{1}{2s}}.
\end{equation}
Note that, quite often, Gevrey regularity is defined locally in space, and not uniformly as in this work. We refer the reader to \cite{MR1249275} for the basics about Gevrey regularity. Here, we generalize this result by proving that the semigroup generated by $\mathcal{P}$ enjoys similar smoothing properties, and we derive a sharp control of the associated seminorms. 

In our context, this regularizing effect is anisotropic and the characteristic directions are given by $(V_k)_{k\geq 0}$ the sequence of nested vector spaces
\begin{equation}\label{01062018E2}
	V_k = \Ran(Q^{\frac12})+\Ran(BQ^{\frac12})+\ldots+\Ran(B^kQ^{\frac12})\subset\mathbb{R}^n,\quad k\geq0,
\end{equation}
where the notation $\Ran$ denotes the range. Assuming that the Kalman rank condition \eqref{10052018E4} holds, we consider $0\le r\le n-1$ the smallest integer satisfying \eqref{05052018E4}. We observe from \eqref{01062018E2} that the following strict inclusions hold :
\begin{equation}\label{30082018E6}
	V_0\subsetneq V_1\subsetneq\ldots\subsetneq V_r=\mathbb{R}^n.
\end{equation}
Moreover, we define $\mathbb P_k$ the orthogonal projection onto the vector subspace $V_k$ for all $0\le k\le r$. All over the work, the orthogonality is taken with respect to the canonical Euclidean structure. We notice from \eqref{30082018E6} that $\mathbb P_r$ is the identity matrix. The following theorem is the main result of this paper and shows that the structure \eqref{30082018E6} of the space $\mathbb R^n$ induced by the family $(V_k)_{0\le k\le r}$ allows one to sharply describe the short-time asymptotics of the regularizing effects induced by the semigroup $(e^{-t\mathcal P})_{t\geq0}$ in the Gevrey type space $G^{\frac1{2s}}(\mathbb R^n)$ : 

\begin{thm}\label{30082018T1} Let $\mathcal{P}$ be the fractional Ornstein-Uhlenbeck operator defined in \eqref{06022018E3} and equipped with the domain \eqref{19062018E1}. When the Kalman rank condition \eqref{10052018E4} holds, there exist some positive constants $C>1$ and $0<t_0<1$ such that for all $k\in\{0,\ldots,r\}$, $q>0$, $0<t<t_0$ and $u\in L^2(\mathbb{R}^n)$,
$$\Vert\langle \mathbb P_kD_x\rangle^q e^{-t\mathcal{P}}u\Vert_{L^2(\mathbb{R}^n)}
\le \frac{C^{1+q}}{t^{q(\frac{1}{2s}+k)}}\ e^{\frac{1}{2}\Tr(B)t}\ q^{\frac{q}{2s}}\ \Vert u\Vert_{L^2(\mathbb{R}^n)},$$
where $\mathbb P_k$ is the orthogonal projection onto the vector subspace $V_k$ defined in \eqref{01062018E2}, $0\le r\le n-1$ is the smallest integer satisfying \eqref{05052018E4} and $\langle \cdot \rangle = \sqrt{1+|\cdot|^2}$ denotes the usual Japanese bracket. In particular, we have that for all $q>0$, $0<t<t_0$ and $u\in L^2(\mathbb{R}^n)$,
$$\Vert\langle D_x\rangle^qe^{-t\mathcal{P}}u\Vert_{L^2(\mathbb{R}^n)}
\le \frac{C^{1+q}}{t^{q(\frac{1}{2s}+r)}}\ e^{\frac{1}{2}\Tr(B)t}\ q^{\frac{q}{2s}}\ \Vert u\Vert_{L^2(\mathbb{R}^n)},$$
since $\mathbb P_r$ is the identity matrix.
\end{thm}

In this statement and all over this work, we denote $D_x = -i\partial_x$. By using the factorial estimate $N^N\le e^NN!$, which holds for any positive integer $N\geq1$, see (0.3.12) in \cite{MR2668420}, we notice that the result of Theorem \ref{30082018T1} implies in particular that there exist some positive constants $C>1$ and $0<t_0<1$ such that for all $k\in\{0,\ldots,r\}$, $N\in\mathbb N$ (the set of all non-negative integers), $0<t<t_0$ and $u\in L^2(\mathbb{R}^n)$,
\begin{equation}\label{06082018E1}
	\Vert\langle\mathbb P_kD_x\rangle^Ne^{-t\mathcal{P}}u\Vert_{L^2(\mathbb{R}^n)}
	\le \frac{C^{1+N}}{t^{N(\frac{1}{2s}+k)}}\ e^{\frac{1}{2}\Tr(B)t}\ (N!)^{\frac1{2s}}\ \Vert u\Vert_{L^2(\mathbb{R}^n)}.
\end{equation}
Therefore, the semigroup $(e^{-t\mathcal P})_{t\geq0}$ is smoothing in the Gevrey space $G^{\frac1{2s}}(\mathbb{R}^n)$, since $\mathbb P_r$ is the identity matrix, with a global control of the seminorms in $\mathcal O(t^{-N(\frac{1}{2s}+r)})$ for small times $t\rightarrow0^+$. This control is sharpened in $\mathcal O(t^{-N(\frac{1}{2s}+k)})$ in the degenerate directions given by the ranges of the matrices $\mathbb P_k$, with $0\le k\le r-1$.

Theorem \ref{30082018T1} can be stated with the matrices $Q^{\frac12}(B^T)^k$ instead of the projections $\mathbb P_k$ :

\begin{cor}\label{28092018C1} Let $\mathcal{P}$ be the fractional Ornstein-Uhlenbeck operator defined in \eqref{06022018E3} and equipped with the domain \eqref{19062018E1}. When the Kalman rank condition \eqref{10052018E4} holds, there exist some positive constants $C>1$ and $0<t_0<1$ such that for all $k\in\{0,\ldots,r-1\}$, $q>0$, $0<t<t_0$ and $u\in L^2(\mathbb{R}^n)$,
$$\Vert\langle Q^{\frac{1}{2}}(B^T)^kD_x\rangle^q e^{-t\mathcal{P}}u\Vert_{L^2(\mathbb{R}^n)}
\le \frac{C^{1+q}}{t^{q(\frac{1}{2s}+k)}}\ e^{\frac{1}{2}\Tr(B)t}\ q^{\frac{q}{2s}}\ \Vert u\Vert_{L^2(\mathbb{R}^n)},$$
and 
$$\Vert\langle D_x\rangle^qe^{-t\mathcal{P}}u\Vert_{L^2(\mathbb{R}^n)}
\le \frac{C^{1+q}}{t^{q(\frac{1}{2s}+r)}}\ e^{\frac{1}{2}\Tr(B)t}\ q^{\frac{q}{2s}}\ \Vert u\Vert_{L^2(\mathbb{R}^n)},$$
where $0\le r\le n-1$ is the smallest integer satisfying \eqref{05052018E4}.
\end{cor}

\begin{rk}\label{03092018R1} Let $\mathcal L$ be the hypoelliptic Ornstein-Uhlenbeck operator defined in \eqref{05052018E1}. Similar properties of regularizing effects for the semigroup $(e^{-t\mathcal L})_{t\geq0}$ on a weighted Lebesgue space $L^2_{\mu}(\mathbb R^n)$, were obtained on the one hand by A. Lunardi in \cite{MR1389786} when the matrix $Q$ is assumed to be positive definite and on the other hand by B. Farkas and A. Lunardi in \cite{MR2257846} and by M. Hitrik, K. Pravda-Starov and J. Viola in \cite{MR3710672} in the degenerate case when the matrix $Q$ is only symmetric positive semidefinite. More precisely, when the semigroup $(e^{-t\mathcal L})_{t\geq0}$ admits an invariant measure $\mu$, which is known to be equivalent \cite{MR1207136} (Section 11.2.3) to the fact all the eigenvalues of the matrix $B$ have a negative real part, the result of \cite{MR3710672} (Corollary 3.3) and a straightforward induction state that there exists a positive constant $C>0$ such that for all $(\alpha,\beta)\in\mathbb N^n$, $0<t<1$ and $u\in L^2_{\mu}(\mathbb R^n)$,
\begin{equation}\label{05092018E2}
	\Vert x^{\alpha}\partial^{\beta}_x(e^{-t\mathcal L}u)\Vert_{L^2_{\mu}(\mathbb R^n)}
	\le\frac{C^{1+\vert\alpha\vert+\vert\beta\vert}}{t^{\vert\alpha+\beta\vert(\frac12+r)}}\ (\alpha!)^{\frac12+r}\ (\beta!)^{\frac12+r}\ \Vert u\Vert_{L^2_{\mu}(\mathbb R^n)},
\end{equation}
where we denote by $L^2_{\mu}(\mathbb R^n)$ the Lebesgue space with weight $\mu$ and $0\le r\le n-1$ the smallest integer satisfying \eqref{05052018E4}. Notice that the index $r$ has the same role in the control of the regularizing effects of the semigroup $(e^{-t\mathcal L})_{t\geq0}$ acting on the weighted space $L^2_{\mu}(\mathbb R^n)$ in \eqref{05092018E2} as in the control of the Gevrey regularizing effects of the semigroup $(e^{-t\mathcal L})_{t\geq0}$ acting on $L^2(\mathbb R^n)$ in \eqref{06082018E1} when $s=1$ and $k=r$.
\end{rk}

\subsection{Null-controllability} In a second step, we study the null-controllability of fractional Ornstein-Uhlenbeck equations posed on the whole space :
\begin{equation}\label{24052018E1}
\left\{\begin{array}{l}
	\partial_tf(t,x) + \mathcal{P} f(t,x) = u(t,x)\mathbbm{1}_{\omega}(x),\quad t>0,\ x\in\mathbb{R}^n, \\[5pt]
	f(0) = f_0\in L^2(\mathbb{R}^n),
\end{array}\right.
\end{equation}
where $\omega\subset\mathbb{R}^n$ is a Borel set with positive Lebesgue measure and $\mathbbm{1}_{\omega}$ is its characteristic function :

\begin{dfn}[Null-controllability] Let $T>0$ and $\omega$ be a Borel subset of $\mathbb{R}^n$ with positive Lebesgue measure. Equation \eqref{24052018E1} is said to be null-controllable from the set $\omega$ in time $T$ if, for any initial datum $f_0\in L^2(\mathbb{R}^n)$, there exists $u\in L^2((0,T)\times\mathbb{R}^n)$, supported in $(0,T)\times\omega$, such that the mild (semigroup) solution of \eqref{24052018E1} satisfies $f(T,\cdot) = 0$.
\end{dfn}

By the Hilbert Uniqueness Method, see \cite{MR2302744} (Theorem 2.44), the null-controllability of the equation \eqref{24052018E1} is equivalent to the observability of the adjoint system
\begin{equation}\label{24052018E2}
\left\{\begin{array}{l}
	\partial_tg(t,x) + \mathcal{P}^*g(t,x) = 0,\quad t>0,\ x\in\mathbb{R}^n, \\[5pt]
	g(0) = g_0\in L^2(\mathbb{R}^n),
\end{array}\right.
\end{equation}
with $\mathcal P^*$ the formal adjoint of the operator $\mathcal P$ in $L^2(\mathbb R^n)$. We recall the definition of observability :

\begin{dfn}[Observability]\label{3} Let $T>0$ and $\omega$ be a Borel subset of $\mathbb{R}^n$ with positive Lebesgue measure. Equation \eqref{24052018E2} is said to be observable from the set $\omega$ in time $T$ if there exists a constant $C(T,\omega)>0$ such that, for any initial datum $g_0\in L^2(\mathbb{R}^n)$, the mild (semigroup) solution of \eqref{24052018E2} satisfies
\begin{gather}\label{24052018E3}
	\Vert g(T,\cdot)\Vert^2_{L^2(\mathbb{R}^n)}\le C(T,\omega)\int_0^T\Vert g(t,\cdot)\Vert^2_{L^2(\omega)}\ dt.
\end{gather}
\end{dfn}

The null-controllability of hypoelliptic Ornstein-Uhlenbeck equations
\begin{equation}\label{25052018E3}
\left\{\begin{array}{l}
	\partial_tf(t,x) - \frac{1}{2}\Tr(Q\nabla^2_x)f(t,x) + \langle Bx,\nabla_x\rangle f(t,x) = u(t,x)\mathbbm{1}_{\omega}(x),\quad t>0,\ x\in\mathbb{R}^n, \\[5pt]
	f(0) = f_0\in L^2(\mathbb{R}^n),
\end{array}\right.
\end{equation}
corresponding to the case when $s=1$, is studied by K. Beauchard and K. Pravda-Starov in \cite{MR3732691} (Theorem 1.3). More precisely, the two authors prove that the equation \eqref{25052018E3} is null-controllable in any positive time, once the control set $\omega\subset\mathbb{R}^n$ satisfies the following geometric condition
\begin{equation}\label{08062018}
	\exists \delta,r>0,\forall y\in\mathbb{R}^n,\exists y'\in\omega,\quad B(y',r)\subset\omega\ \text{and}\ \vert y-y'\vert<\delta.
\end{equation}
Their proof is based on a Lebeau-Robbiano strategy, that is, on the combinaison of a so-called spectral inequality and a dissipation estimate for the high-frequencies of the evolution operator, that we shall also be using in this work.

The case where $B = 0_n$ and $Q = 2^{\frac{1}{s}}I_n$, corresponding to the fractional heat equation
\begin{equation}\label{24052018E4}
\left\{\begin{aligned}
	& \partial_tf(t,x) +  (- \Delta_x)^s f(t,x) = u(t,x)\mathbbm{1}_{\omega}(x),\quad t>0,\ x\in\mathbb{R}^n, \\
	& f(0) = f_0\in L^2(\mathbb{R}^n),
\end{aligned}\right.
\end{equation}
is widely studied. When $0<s<1/2$, A. Koenig proved in \cite{K} (Theorem 3) that this equation is not null-controllable in any positive time, once $\omega\subset\mathbb R^n$ is open with $\omega\ne\mathbb R^n$. Furthermore, no positive null-controllability result is known with non trivial measurable control supports for such $s$. This is not the case when $s>1/2$, since then, L. Miller derived in \cite{MR2272076} the null-controllability in any positive time of  \eqref{24052018E4} for control subsets $\omega\subset\mathbb{R}^n$ which are exteriors of compacts sets, see Subsection 3.2, and more specifically Theorem 3.1, in \cite{MR2272076}. A. Koenig also studied this equation for $s=1/2$ in \cite{MR3730500}, but only on the one-dimensional torus $\mathbb T$ and proves that when $\omega = \mathbb T\setminus[a,b]$, with $[a,b]$ a non-trivial segment of $\mathbb T$, the equation \eqref{24052018E4} is not null-controllable in time $T$ for all $T>0$. The case when $s=1$ corresponding to the heat equation is now fully understood with the recent results by M. Egidi and I. Veselic in \cite{MR3816981} and G. Wang, M. Wang, C. Zhang and Y. Zhang in \cite{WZ} establishing that the heat equation posed on the whole Euclidean space is null-controllable in any positive time if and only if the control subset $\omega\subset\mathbb{R}^n$ is thick. The thickness of a subset of $\mathbb{R}^n$ is defined as follows :

\begin{dfn} Let $\gamma\in(0,1]$ and $a = (a_1,\ldots,a_n)\in(\mathbb{R}^*_+)^n$. Let $C = [0,a_1]\times\ldots\times[0,a_n]\subset\mathbb{R}^n$. A set $\omega\subset\mathbb{R}^n$ is called $(\gamma,a)$-thick if it is measurable and 
$$\forall x\in\mathbb{R}^n,\quad \vert\omega\cap(x+C)\vert\geq\gamma\prod_{j=1}^na_j,$$
where $\vert\omega\cap(x+C)\vert$ stands for the Lebesgue measure of $\omega\cap(x+C)$. A set $\omega\subset\mathbb{R}^n$ is called thick if there exist $\gamma\in(0,1]$ and $a\in(\mathbb{R}^*_+)^n$ such that $\omega$ is $(\gamma,a)$-thick.
\end{dfn}

\noindent Note that the thickness is weaker than the condition \eqref{08062018} considered in \cite{MR3732691}.

By taking advantage of the smoothing effect of the semigroup $(e^{-t\mathcal{P}})_{t\geq0}$, we aim in this work at proving that the fractional Ornstein-Uhlenbeck equation \eqref{24052018E1} is null-controllable in any positive time from thick control subsets of $\mathbb{R}^n$, once $s>1/2$ : 

\begin{thm}\label{05022018T3} Let $\mathcal{P}$ be the fractional Ornstein-Uhlenbeck operator defined in \eqref{06022018E3} and equipped with the domain \eqref{19062018E1}. We assume that $s>1/2$, and that the Kalman rank condition \eqref{10052018E4} holds. If $\omega\subset\mathbb{R}^n$ is a thick set, then the parabolic equation
\begin{equation}\label{25052018E1}
\left\{\begin{array}{l}
	\partial_tf(t,x) + \mathcal{P} f(t,x) = u(t,x)\mathbbm{1}_{\omega}(x),\quad t>0,\ x\in\mathbb{R}^n, \\[5pt]
	f(0) = f_0\in L^2(\mathbb{R}^n),
\end{array}\right.
\end{equation}
is null-controllable from the set $\omega$ in any positive time $T>0$.
\end{thm}

\begin{rk} Theorem \ref{05022018T3} extends Theorem 1.3 in \cite{MR3732691} (concerning the null-controllability of hypoelliptic Ornstein-Uhlenbeck equations) to thick control subsets, and the result by L. Miller mentioned above since it implies that the fractional heat equation \eqref{24052018E4} is null-controllable in any positive time, once $s>1/2$ and the control subset $\omega$ is thick.
\end{rk}

\begin{rk} It is an open and interesting issue to know if the evolution equations \eqref{25052018E1} associated with fractional Ornstein-Uhlenbeck operators are null-controllable or not in the case when $0<s\le1/2$. As mentioned above, the only existing result in this case is A. Koenig's one \cite{K} (Theorem 3) concerning the non null-controllability of the fractional heat equation \eqref{24052018E4} from strict open subsets $\omega$ of $\mathbb R^n$ when $0<s<1/2$.
\end{rk}

\begin{ex} If $\omega\subset\mathbb{R}^n$ is thick, the fractional Kolmogorov equation posed on the whole space
$$\left\{\begin{array}{l}
	\partial_tf(t,x,v) + v\cdot(\nabla_xf)(t,x,v) + (-\Delta_v)^sf(t,x,v) = u(t,x,v)\mathbbm{1}_{\omega}(x,v),\quad t>0,\ (x,v)\in\mathbb{R}^{2n}, \\[5pt]
	f(0) = f_0\in L^2(\mathbb{R}^{2n}),
\end{array}\right.$$
where $s>1/2$, is null-controllable in any positive time $T>0$, since the matrices defined in \eqref{08062018E3} satisfy the Kalman rank condition \eqref{10052018E4}.
\end{ex}

By using the change of unknows $g = e^{-\frac{1}{2}\Tr(B)t}f$ and $v = e^{-\frac{1}{2}\Tr(B)t}u$, where $f$ is a solution of \eqref{25052018E1} with control $u$, we notice that the result of Theorem \ref{05022018T3} is equivalent to the null-controllability of the equation
\begin{equation}\label{12022018E5}
\left\{\begin{array}{l}
	\partial_tg(t,x) + \mathcal{P}_{co}g(t,x) = v(t,x)\mathbbm{1}_{\omega}(x), \\[5pt]
	g(0) = f_0\in L^2(\mathbb{R}^n),
\end{array}\right.
\end{equation}
where 
\begin{equation}\label{04052018E1}
	\mathcal{P}_{co} = \frac{1}{2}\Tr^s(-Q\nabla^2_x) + \langle Bx,\nabla_x\rangle + \frac{1}{2}\Tr(B).
\end{equation}
We prove in Corollary \ref{16032018C2} that the adjoint of the operator $\mathcal{P}_{co}$ equipped with the domain 
\begin{equation}\label{04052018E2}
	D(\mathcal{P}_{co}) = \{u\in L^2(\mathbb{R}^n),\quad \mathcal{P}_{co}u\in L^2(\mathbb{R}^n)\} = D(\mathcal{P})
\end{equation}
is given by
$$(\mathcal{P}_{co})^* = \frac{1}{2}\Tr^s(-Q\nabla^2_x) + \langle -Bx,\nabla_x\rangle + \frac{1}{2}\Tr(-B),$$
with domain $D(\mathcal P)$. Moreover, $-B$ and $Q^{\frac{1}{2}}$ also satisfy the Kalman rank condition \eqref{10052018E4}. Therefore, by the Hilbert Uniqueness Method, the result of null-controllability given by Theorem \ref{05022018T3} is equivalent to the following observability estimate :

\begin{thm}\label{23052018T1} Let $\mathcal{P}_{co}$ be the operator defined in \eqref{04052018E1} and equipped with the domain \eqref{04052018E2}. We assume that $s>1/2$, and that the Kalman rank condition \eqref{10052018E4} holds. If $\omega\subset\mathbb{R}^n$ is a thick set, there exists a positive constant $C>0$ such that for all $T>0$ and $u\in L^2(\mathbb{R}^n)$,
\begin{equation}\label{06022018E2}
	\Vert e^{-T\mathcal{P}_{co}}u\Vert^2_{L^2(\mathbb{R}^n)}\le C\exp\left(\frac{C}{T^{\frac{1+2rs}{2s-1}}}\right)\int_0^T\Vert e^{-t\mathcal{P}_{co}}u\Vert^2_{L^2(\omega)}\ dt,
\end{equation}
where $0\le r\le n-1$ is the smallest integer satisfying \eqref{05052018E4}.
\end{thm}

\begin{rk} Let $s>1/2$ and $\omega\subset\mathbb R^n$ be a $(\gamma,a)$-thick set. Theorem \ref{23052018T1} applied in the case of the fractional heat equation states that there exists a positive constant $C>1$ such that for all $T>0$ and $u\in L^2(\mathbb R^n)$,
\begin{equation}\label{14062020E1}
	\Vert e^{-T(-\Delta)^s}u\Vert^2_{L^2(\mathbb{R}^n)}\le C\exp\bigg(\frac C{T^{\frac1{2s-1}}}\bigg)\int_0^T\Vert e^{-t(-\Delta)^s}u\Vert^2_{L^2(\omega)}\ dt.
\end{equation}
Observability estimates for this semigroup in any positive times have already been established in the work \cite{NTTV} (Theorem 3.8), where the authors proved that there exist some positive constants $C_1,C_2,C_3,C_4>0$ such that for all $T>0$ and $u\in L^2(\mathbb R^n)$,
\begin{equation}\label{14062020E2}
	\Vert e^{-T(-\Delta)^s}u\Vert^2_{L^2(\mathbb{R}^n)}\le C^2_{obs}\int_0^T\Vert e^{-t(-\Delta)^s}u\Vert^2_{L^2(\omega)}\ dt,
\end{equation}
the positive constant $C_{obs}^2>0$ being given by
$$C^2_{obs} = \frac{C_1}{\gamma^{C_2n}T}\exp\bigg(\frac{C_3(\vert a\vert_1\ln(C_4^n/\gamma))^{\frac{2s}{2s-1}}}{T^{\frac1{2s-1}}}\bigg),$$
with $\vert a\vert_1 = a_1+\ldots+a_n$. Notice that the constant $C>1$ appearing in the estimate \eqref{14062020E1} could also be expressed in terms of $s>1/2$, $0<\gamma\le 1$ and $a\in(\mathbb R^*_+)^n$. Up to the value of this constant $C>1$, these two observability estimates are equivalent.
\end{rk}

The result of Theorem \ref{05022018T3} can be refined when the operator $\mathcal P$ stands for the fractional Laplacian. Indeed, for all $s>0$, the thickness of the control set $\omega\subset\mathbb R^n$ turns out to be also a necessary condition for the null-controllability of the fractional heat equation \eqref{24052018E4} :

\begin{thm}\label{20082019T1} Let $s>0$, $T>0$ be some positive real numbers and $\omega\subset\mathbb{R}^n$ be a measurable set. If the fractional heat equation \eqref{24052018E4} is null-controllable from the set $\omega$ in time $T>0$, then $\omega$ is thick.
\end{thm}

\begin{rk} Let $s>1/2$ and $\omega\subset\mathbb R^n$. The results of Theorem \ref{05022018T3} and Theorem \ref{20082019T1} imply that for all positive time $T>0$, the fractional heat equation \eqref{24052018E4} is null-controllable from the set $\omega$ in time $T$ if and only if $\omega$ is thick. Therefore, these two theorems extend the result \cite{MR3816981} (Theorem 3) by M. Egidi and I. Veselic.
\end{rk}

\subsection{Global subelliptic estimates} Finally, we derive global $L^2$ subelliptic estimates for the fractional Ornstein-Uhlenbeck operator $\mathcal{P}$. Let $\mathcal{L}$ be the hypoelliptic Ornstein-Uhlenbeck operator defined in \eqref{05052018E1}. In the work \cite{MR2729292}, M. Bramanti, G. Cupini, E. Lanconelli and E. Priola proved global $L^p$ estimates for $\mathcal{L}$, with $1<p<\infty$. More specifically, they showed that for every $1<p<\infty$, there exists a positive constant $C_p>0$ such that for all $u\in C^{\infty}_0(\mathbb{R}^n)$,
$$\Vert\vert\mathbb P_0D_x\vert^2 u\Vert_{L^p(\mathbb{R}^n)}\le C_p\left[\Vert \mathcal{L}u\Vert_{L^p(\mathbb{R}^n)} +\Vert u\Vert_{L^p(\mathbb{R}^n)}\right],$$
where the operator $\vert\mathbb P_0D_x\vert^2$ denotes the Fourier multiplier associated to the symbol $\vert\mathbb P_0\xi\vert^2$, with $\mathbb P_0$ the orthogonal projection onto $V_0$ the vector space defined in \eqref{01062018E2}. This result provides global $L^p$ estimates of the elliptic frequency directions, since the standard symbol of $\mathcal{L}$ is given by
\begin{equation}\label{07052018E2}
	\vert Q^{\frac{1}{2}}\xi\vert^2 + i\langle Bx,\xi\rangle,\quad (x,\xi)\in\mathbb{R}^{2n}.
\end{equation}
However, the result of Bramanti and al. does not provide any control for the degenerate frequency directions $\mathbb{R}^{2n}\setminus V_0$. Despite the operator $\mathcal{L}$ may fail to be elliptic, its hypoelliptic properties induced by the Kalman rank condition allow to expect that some controls of the non-elliptic frequency directions still hold.

In this work, we consider specifically the $L^2$ case and we aim at establishing global $L^2$ subelliptic estimates for $\mathcal{P}$ in all frequency directions :

\begin{thm}\label{28092018T1} Let $\mathcal{P}$ be the fractional Ornstein-Uhlenbeck operator defined in \eqref{06022018E3} and equipped with the domain \eqref{19062018E1}. When the Kalman rank condition \eqref{10052018E4} holds, there exists a positive constant $c>0$ such that for all $u\in D(\mathcal{P})$,
$$\sum_{k=0}^r\Vert\langle\mathbb P_kD_x\rangle^{\frac{2s}{1+2ks}}u\Vert_{L^2(\mathbb{R}^n)}\le c\left[\Vert \mathcal{P}u\Vert_{L^2(\mathbb{R}^n)} + \Vert u\Vert_{L^2(\mathbb{R}^n)}\right],$$
where $\mathbb P_k$ is the orthogonal projection onto the vector subspace $V_k$ defined in \eqref{01062018E2} and $0\le r\le n-1$ is the smallest integer satisfying \eqref{05052018E4}.
\end{thm}

This result can be reformulated with the matrices $Q^{\frac12}(B^T)^k$ :

\begin{cor}\label{03102018C1} Let $\mathcal{P}$ be the fractional Ornstein-Uhlenbeck operator defined in \eqref{06022018E3} and equipped with the domain \eqref{19062018E1}. When the Kalman rank condition \eqref{10052018E4} holds, there exists a positive constant $c>0$ such that for all $u\in D(\mathcal{P})$,
$$\sum_{k=0}^{r-1}\Vert\langle Q^{\frac{1}{2}}(B^T)^kD_x\rangle^{\frac{2s}{1+2ks}}u\Vert_{L^2(\mathbb{R}^n)} + \Vert\langle D_x\rangle^{\frac{2s}{1+2rs}}u\Vert_{L^2(\mathbb{R}^n)}
\le c\left[\Vert \mathcal{P}u\Vert_{L^2(\mathbb{R}^n)} + \Vert u\Vert_{L^2(\mathbb{R}^n)}\right],$$
where $0\le r\le n-1$ is the smallest integer satisfying \eqref{05052018E4}.
\end{cor}

When the parameter $2s$ is a positive integer, this result proves that $\mathcal{P}$ enjoys a global $L^2$ subelliptic estimate 
$$\Vert\langle D_x\rangle^{2s(1 - \delta)}u\Vert_{L^2(\mathbb{R}^n)}\lesssim\Vert \mathcal{P}u\Vert_{L^2(\mathbb{R}^n)} + \Vert u\Vert_{L^2(\mathbb{R}^n)},$$
with loss of
$$\delta = \frac{2rs}{1+2rs}>0,$$
derivatives compared to the elliptic case. 

\begin{rk} Notice that Corollary \ref{03102018C1} extends the result of Theorem 1.4 in \cite{MR3710672} to a class of accretive quadratic operators associated to quadratic forms with non-trivial singular spaces (see e.g. (1.7) in \cite{MR3710672} for the definition of this notion). Indeed, let us consider $q^w(x,D_x)$ the quadratic operator defined as the Weyl quantization of the quadratic form $q:\mathbb{R}^n_x\times\mathbb{R}^n_{\xi}\rightarrow\mathbb{C}$ defined on the phase space by \eqref{07052018E2}. An immediate computation shows that the singular space of $q$ is $S = \mathbb{R}^n\times\{0\}$, see e.g. \cite{MR3732691} (p.11). Since $q^w(x,D_x) = \mathcal{L} + \frac{1}{2}\Tr(B)$, it follows from Corollary \ref{03102018C1} that there exists a positive constant $c>0$ such that for all $u\in D(\mathcal{L})$,
$$\Vert\Lambda_0^2u\Vert_{L^2(\mathbb{R}^n)} + \sum_{k=1}^r\Vert\Lambda_k^{\frac{2}{2k+1}}u\Vert_{L^2(\mathbb{R}^n)}
\le c\left[\Vert q^w(x,D_x)u\Vert_{L^2(\mathbb{R}^n)} + \Vert u\Vert_{L^2(\mathbb{R}^n)}\right],$$
with the notations of p.624 in \cite{MR3710672}. In addition, it conforts the conjecture made in \cite{MR3710672} that the power over the operator $\Lambda_0$ in \cite{MR3710672} (Theorem 1.4) should be $2$.
\end{rk}

\begin{rk} In the work \cite{MR2257846}, B. Farkas and A. Lunardi provided a sharp embedding for the domain of the hypoelliptic Ornstein-Uhlenbeck operator $\mathcal L$ with invariant measures acting on $L^2$ spaces, in some anisotropic weighted Sobolev spaces. We will not recall here this result in detail but we point out that the regularity exponents that define the anisotropic weighted Sobolev spaces in question are given by $2/(1+2k)$, with $0\le k\le r$ and $0\le r\le n-1$ the smallest integer satisfying \eqref{05052018E4}, and that these exponents are also the regularity exponents appearing in Theorem \ref{28092018T1} applied with $\mathcal P = \mathcal L$, corresponding to the case when $s=1$.
\end{rk}

An immediate consequence of Corollary \ref{03102018C1} is the following :

\begin{cor}\label{24052018C1} Let $\mathcal{P}$ be the fractional Ornstein-Uhlenbeck operator defined in \eqref{06022018E3} and equipped with the domain \eqref{19062018E1}. When the Kalman rank condition \eqref{10052018E4} holds, there exists a positive constant $c>0$ such that for all $u\in D(\mathcal{P})$,
$$\Vert\langle Bx,\nabla_x\rangle u\Vert_{L^2(\mathbb{R}^n)}
\le c\left[\Vert \mathcal{P}u\Vert_{L^2(\mathbb{R}^n)} + \Vert u\Vert_{L^2(\mathbb{R}^n)}\right].$$
\end{cor}

\begin{ex} Since the matrices defined by \eqref{08062018E3} satisfy the Kalman rank condition \eqref{10052018E4}, and the associated smallest integer satisfying \eqref{05052018E4} is $r=1$, there exists a positive constant $c>0$ such that for all $u\in D(\mathcal{P})$,
$$\Vert v\cdot(\nabla_x u)\Vert_{L^2(\mathbb{R}^{2n})} + \Vert\langle D_x\rangle^{\frac{2s}{1+2s}}u\Vert_{L^2(\mathbb{R}^{2n})} + \Vert\langle D_v\rangle^{2s}u\Vert_{L^2(\mathbb{R}^{2n})}
\le c\left[\Vert \mathcal{P}u\Vert_{L^2(\mathbb{R}^{2n})} + \Vert u\Vert_{L^2(\mathbb{R}^{2n})}\right],$$
where $\mathcal{P}$ stands for the fractional Kolmogorov operator defined by
$$\mathcal{P} = v\cdot\nabla_x + (-\Delta_v)^s,\quad (x,v)\in\mathbb{R}^{2n}.$$
\end{ex}

\subsubsection*{Outline of the work} In Section \ref{sec_GreofOUs}, we establish the Gevrey smoothing effects for semigroups generated by fractional Ornstein-Uhlenbeck operators under the Kalman rank condition, after checking in Section \ref{sec_frac_OUO} that these semigroups are well-defined. Thanks to these regularizing effects, we study the null-controllability of fractional Ornstein-Uhlenbeck equations in Section \ref{sec_OeffOUs}, and $L^2$ subelliptic estimates enjoyed by fractional Ornstein-Uhlenbeck operators in Section \ref{sec_GLseffOUo}. Section \ref{sec_appendix} is an appendix containing the proofs of some technical results.

\subsubsection*{Notations} The following notations and conventions will be used all over the work :
\begin{enumerate}[label=\textbf{\arabic*.},leftmargin=* ,parsep=2pt,itemsep=0pt,topsep=2pt]
\item The canonical Euclidean scalar product of $\mathbb R^n$ is denoted by $\langle\cdot,\cdot\rangle$ and $\vert\cdot\vert$ stands for the associated canonical Euclidean norm.
\item For all measurable subset $\omega\subset\mathbb R^n$, the inner product of $L^2(\omega)$ is defined for all $u,v\in L^2(\omega)$ by
$$\langle u,v\rangle_{L^2(\omega)} = \int_{\omega}u(x)\overline{v(x)}\ dx,$$
while $\Vert\cdot\Vert_{L^2(\omega)}$ stands for its associated norm.
\item For all function $u\in\mathscr{S}(\mathbb{R}^n)$, the Fourier transform of $u$ is denoted by $\widehat{u}$ or $\mathscr{F}(u)$ while $\mathscr{F}^{-1}(u)$ stands for its inverse Fourier transform and $\mathscr{F}(u)$, $\mathscr{F}^{-1}(u)$ are respectively defined by
$$\widehat{u}(\xi) = \mathscr{F}(u)(\xi) = \int_{\mathbb R^n}e^{-i\langle x,\xi\rangle}u(x)\ dx\quad \text{and}\quad 
\mathscr{F}^{-1}(u)(x) = \frac1{(2\pi)^n}\int_{\mathbb R^n}e^{i\langle x,\xi\rangle}\widehat u(\xi)\ d\xi.$$
With this convention, the Plancherel theorem states that 
$$\forall u\in\mathscr{S}(\mathbb{R}^n),\quad \Vert\widehat u\Vert^2_{L^2(\mathbb{R}^n)} = (2\pi)^n\Vert u\Vert^2_{L^2(\mathbb{R}^n)}.$$
\item The Japanese bracket $\langle\cdot\rangle$ is defined for all $x\in\mathbb R^n$ by $\langle x\rangle = (1+\vert x\vert^2)^{\frac12}$.
\item For all real $n\times n$ matrix $M\in M_n(\mathbb R)$ and all non-negative real number $q\geq0$, $\vert MD_x\vert^q$ and $\langle MD_x\rangle^q$ are the Fourier multipliers associated respectively to the symbols $\vert M\xi\vert^q$ and $\langle M\xi\rangle^q$.
\item For all measurable subset $\omega\subset\mathbb R^n$, $\mathbbm{1}_{\omega}$ stands for the characteristic function of $\omega$.
\end{enumerate}

\section{Fractional Ornstein-Uhlenbeck operators}
\label{sec_frac_OUO}
This section is devoted to the proof of Theorem \ref{06022018T1}.

\subsection{Graph approximation} We begin by studying the graphs of fractional Ornstein-Uhlenbeck operators. We prove that the Schwartz space $\mathscr{S}(\mathbb{R}^n)$ is dense in their domains equipped with the graph norm, by using the classical symbolic calculus and convolution estimates. Then, we compute the adjoints of these operators, and we study their positivity.

\begin{prop}\label{15032018P3} Let $\mathcal{P}$ be the fractional Ornstein-Uhlenbeck operator defined in \eqref{06022018E3} and equipped with the domain \eqref{19062018E1}. Then, for all $u\in D(\mathcal{P})$, there exists $(u_k)_k$ a sequence of $\mathscr{S}(\mathbb{R}^n)$ such that
$$\lim_{k\rightarrow+\infty}u_k = u\quad \text{and}\quad \lim_{k\rightarrow+\infty}\mathcal{P}u_k = \mathcal{P}u\quad \text{in $L^2(\mathbb{R}^n)$}.$$
\end{prop}

\begin{proof} For all $k\geq1$, we consider the pseudodifferential operator
\begin{gather}\label{09052018E1}
	\psi_k(x,D_x) = \psi\left(\frac{\vert x\vert^2}{k^2}\right)\psi\left(\frac{D^2_x}{k^{2\alpha}}\right),
\end{gather}
where $D_x = -i\partial_x$, $\psi\in C^{\infty}_0(\mathbb{R})$ is such that $0\le\psi\le1$, $\psi = 1$ on $[-1,1]$ and $\Supp\psi\subset[-2,2]$, and $\alpha>0$ is a positive constant satisfying $(2s-1)\alpha<1$. Notice that for all $k\geq1$, the pseudodifferential operator $\psi_k(x,D_x)$ is acting on $L^2(\mathbb R^n)$ in the following way
$$\forall u\in L^2(\mathbb R^n),\quad \psi_k(x,D_x)u = \frac1{(2\pi)^n}\psi\left(\frac{\vert x\vert^2}{k^2}\right)\int_{\mathbb R^n}e^{i\langle x,\xi\rangle}\psi\left(\frac{\vert\xi\vert^2}{k^{2\alpha}}\right)\widehat u(\xi)\ d\xi.$$
Since $\psi$ is compactly supported, we get that
\begin{equation}\label{23052018E1}
	\forall k\geq1,\quad \psi_k(x,D_x) : L^2(\mathbb{R}^n)\rightarrow C^{\infty}_0(\mathbb{R}^n)\subset\mathscr{S}(\mathbb{R}^n).
\end{equation}
Let us first check that
\begin{gather}\label{15032018E7}
	\forall k\geq1, \forall u\in L^2(\mathbb{R}^n),\quad \lim_{k\rightarrow+\infty}\psi_k(x,D_x)u = u\quad \text{in $L^2(\mathbb{R}^n)$}.
\end{gather}
Let $u\in L^2(\mathbb{R}^n)$. We have that for all $k\geq1$,
\begin{align*}
	&\ \Vert \psi_k(x,D_x)u-u\Vert_{L^2(\mathbb{R}^n)} \\[5pt]
	\le &\ \Vert\psi\left(\frac{\vert x\vert^2}{k^2}\right)\psi\left(\frac{D^2_x}{k^{2\alpha}}\right) u - \psi\left(\frac{\vert x\vert^2}{k^2}\right)u\Vert_{L^2(\mathbb{R}^n)}
	+ \Vert \psi\left(\frac{\vert x\vert^2}{k^2}\right)u - u\Vert_{L^2(\mathbb{R}^n)} \\[5pt]
	\le &\ \Vert\psi\left(\frac{D^2_x}{k^{2\alpha}}\right) u - u\Vert_{L^2(\mathbb{R}^n)} + \Vert \psi\left(\frac{\vert x\vert^2}{k^2}\right)u - u\Vert_{L^2(\mathbb{R}^n)}
	\underset{k\rightarrow+\infty}{\longrightarrow} 0,
\end{align*}
from Lebesgue's dominated convergence theorem and Plancherel's theorem. Thus, \eqref{15032018E7} is proved. Now, we consider $u\in D(\mathcal{P})$ and $u_k = \psi_k(x,D_x)u$ for all $k\geq1$. As a consequence of \eqref{23052018E1} and \eqref{15032018E7}, $(u_k)_k$ is a sequence of Schwartz functions that converges to $u$ in $L^2(\mathbb{R}^n)$. Since $\mathcal{P}u\in L^2(\mathbb{R}^n)$ by definition of $D(\mathcal{P})$, we can apply once again \eqref{15032018E7} to get that
$$\lim_{k\rightarrow+\infty}\psi_k(x,D_x)\mathcal{P}u = \mathcal{P}u\quad \text{in $L^2(\mathbb{R}^n)$}.$$
If the operators $\psi_k(x,D_x)$ and $\mathcal{P}$ were commutative, the proposition would be proven. It is not the case but to conclude, it is sufficient to check that 
\begin{gather}\label{09052018E9}
	\lim_{k\rightarrow+\infty}[\mathcal{P},\psi_k(x,D_x)]u = 0\quad \text{in $L^2(\mathbb{R}^n)$}.
\end{gather}
We write
$$[\mathcal{P},\psi_k(x,D_x)] = [\langle Bx,\nabla_x\rangle,\psi_k(x,D_x)] + \frac{1}{2}[\Tr^s(-Q\nabla^2_x),\psi_k(x,D_x)],$$
and consider the two commutators separately. \\[5pt]
\textbf{1.} By definition of the commutator, 
\begin{equation}\label{18062018E7}
	[\langle Bx,\nabla_x\rangle,\psi_k(x,D_x)]u = \langle Bx,\nabla_x(\psi_k(x,D_x)u)\rangle - \psi_k(x,D_x)\langle Bx,\nabla_xu\rangle.
\end{equation}
First, we notice that
\begin{align}\label{18062018E1}
	&\ \langle Bx,\nabla_x(\psi_k(x,D_x)u)\rangle = \langle Bx,\nabla_x\left[\psi\left(\frac{\vert x\vert^2}{k^2}\right)\psi\left(\frac{D^2_x}{k^{2\alpha}}\right)u\right]\rangle \\[5pt]
	= &\ \langle Bx,\nabla_x\left[\psi\left(\frac{\vert x\vert^2}{k^2}\right)\right]\rangle\psi\left(\frac{D^2_x}{k^{2\alpha}}\right)u
	+ \psi\left(\frac{\vert x\vert^2}{k^2}\right)\langle Bx,\psi\left(\frac{D^2_x}{k^{2\alpha}}\right)\nabla_xu\rangle \nonumber \\[5pt]
	= &\ \frac{2}{k^2}\langle Bx,x\rangle\psi'\left(\frac{\vert x\vert^2}{k^2}\right)\psi\left(\frac{D^2_x}{k^{2\alpha}}\right)u
	+ \psi\left(\frac{\vert x\vert^2}{k^2}\right)\langle Bx,\psi\left(\frac{D^2_x}{k^{2\alpha}}\right)\nabla_xu\rangle. \nonumber
\end{align}
The last term of the previous equality also writes as
\begin{multline}\label{18062018E8}
	\psi\left(\frac{\vert x\vert^2}{k^2}\right)\langle Bx,\psi\left(\frac{D^2_x}{k^{2\alpha}}\right)\nabla_xu\rangle \\[5pt]
	= - \psi\left(\frac{\vert x\vert^2}{k^2}\right)\mathscr{F}^{-1}\left(\frac{2}{k^{2\alpha}}\langle B^T\xi,\xi\rangle\psi'\left(\frac{\vert\xi\vert^2}{k^{2\alpha}}\right)\widehat{u}\right)
+ \psi_k(x,D_x)\langle Bx,\nabla_xu\rangle.
\end{multline}
Indeed, it follows from a direct computation that
\begin{align*}
	&\ \mathscr{F}\left(\langle Bx,\psi\left(\frac{D^2_x}{k^{2\alpha}}\right)\nabla_xu\rangle\right)
	= \langle iB\nabla_{\xi},\psi\left(\frac{\vert\xi\vert^2}{k^{2\alpha}}\right)i\xi\widehat{u}\rangle 
	=  - \langle \nabla_{\xi},\psi\left(\frac{\vert\xi\vert^2}{k^{2\alpha}}\right)B^T\xi\widehat{u}\rangle \\[5pt]
	= &\ -\langle B^T\xi,\nabla_{\xi}\left[\psi\left(\frac{\vert\xi\vert^2}{k^{2\alpha}}\right)\widehat{u}\right]\rangle 
	- \psi\left(\frac{\vert\xi\vert^2}{k^{2\alpha}}\right)\langle\nabla_{\xi},B^T\xi\rangle\widehat{u} \\[5pt]
	= &\ -\langle B^T\xi,\nabla_{\xi}\left[\psi\left(\frac{\vert\xi\vert^2}{k^{2\alpha}}\right)\right]\rangle \widehat{u}
	- \psi\left(\frac{\vert\xi\vert^2}{k^{2\alpha}}\right)\left[\langle B^T\xi,\nabla_{\xi}\widehat{u}\rangle
	+ \Tr(B)\widehat{u}\right] \\[5pt]
	= &\ -\frac{2}{k^{2\alpha}}\langle B^T\xi,\xi\rangle\psi'\left(\frac{\vert\xi\vert^2}{k^{2\alpha}}\right)\widehat{u}
	- \psi\left(\frac{\vert\xi\vert^2}{k^{2\alpha}}\right)\left[\langle B^T\xi,\nabla_{\xi}\widehat{u}\rangle
	+ \Tr(B)\widehat{u}\right],
\end{align*}
that is
$$\langle Bx,\psi\left(\frac{D^2_x}{k^{2\alpha}}\right)\nabla_xu\rangle =
-\mathscr{F}^{-1}\left(\frac{2}{k^{2\alpha}}\langle B^T\xi,\xi\rangle\psi'\left(\frac{\vert\xi\vert^2}{k^{2\alpha}}\right)\widehat{u}\right)
+ \psi\left(\frac{D^2_x}{k^{2\alpha}}\right)\langle Bx,\nabla_xu\rangle,$$
since we also have that
\begin{multline*}
	\mathscr{F}\left(\langle Bx,\nabla_xu\rangle\right) 
	= \langle iB\nabla_{\xi},i\xi\widehat{u}\rangle 
	= -\langle\nabla_{\xi},B^T\xi\widehat{u}\rangle \\
 	= -\langle B^T\xi,\nabla_{\xi}\widehat u\rangle - \langle\nabla_{\xi},B^T\xi\rangle\widehat{u} 
	= -\langle B^T\xi,\nabla_{\xi}\widehat u\rangle - \Tr(B)\widehat{u}.
\end{multline*}
It follows from \eqref{09052018E1}, \eqref{18062018E7}, \eqref{18062018E1} and \eqref{18062018E8} that for all $k\geq1$,
\begin{multline}\label{18062018E10}
	[\langle Bx,\nabla_x\rangle,\psi_k(x,D_x)]u \\
	= \frac{2}{k^2}\langle Bx,x\rangle\psi'\left(\frac{\vert x\vert^2}{k^2}\right)\psi\left(\frac{D^2_x}{k^{2\alpha}}\right)u
	- \psi\left(\frac{\vert x\vert^2}{k^2}\right)\mathscr{F}^{-1}\left(\frac{2}{k^{2\alpha}}\langle B^T\xi,\xi\rangle\psi'\left(\frac{\vert\xi\vert^2}{k^{2\alpha}}\right)\widehat{u}\right).
\end{multline}
Now, let us prove the following convergence
\begin{equation}\label{18062018E2}
	\lim_{k\rightarrow+\infty}\Vert\frac{2}{k^2}\langle Bx,x\rangle\psi'\left(\frac{\vert x\vert^2}{k^2}\right)\psi\left(\frac{D^2_x}{k^{2\alpha}}\right)u\Vert_{L^2(\mathbb{R}^n)} = 0.
\end{equation}
On the one hand, since $\psi'$ is bounded and $\psi'(0) = 0$, we get by homogeneity that
\begin{equation}\label{18062018E3}
	\sup_{k\geq1}\Vert\frac{2}{k^2}\langle Bx,x\rangle\psi'\left(\frac{\vert x\vert^2}{k^2}\right)\Vert_{L^{\infty}(\mathbb{R}^n)}<+\infty,
\end{equation}
and 
\begin{equation}\label{18062018E4}
	\forall x\in\mathbb{R}^n,\quad \frac{2}{k^2}\langle Bx,x\rangle\psi'\left(\frac{\vert x\vert^2}{k^2}\right)\underset{k\rightarrow+\infty}{\longrightarrow}0.
\end{equation}
On the other hand, the following convergence
$$\lim_{k\rightarrow+\infty}\psi\left(\frac{D^2_x}{k^{2\alpha}}\right)u = u\quad \text{in $L^2(\mathbb{R}^n)$},$$
and the classical corollary of the Riesz-Fischer theorem, see e.g. Theorem IV.9 in \cite{MR2759829}, prove that up to an extraction, 
\begin{equation}\label{18062018E5}
	\psi\left(\frac{D^2_x}{k^{2\alpha}}\right)u\underset{k\rightarrow+\infty}{\longrightarrow}u\quad \text{a.e. on $\mathbb{R}^n$},
\end{equation}
and give the existence of $v\in L^2(\mathbb{R}^n)$ such that for all $k$,
\begin{equation}\label{18062018E6}
	\vert\psi\left(\frac{D^2_x}{k^{2\alpha}}\right)u\vert\le\vert v\vert.
\end{equation}
Then, \eqref{18062018E2} is a consequence of \eqref{18062018E3}, \eqref{18062018E4}, \eqref{18062018E5}, \eqref{18062018E6} and the dominated convergence theorem. By arguing in the very same way, we derive that
$$\lim_{k\rightarrow+\infty}\Vert\frac{2}{k^{2\alpha}}\langle B^T\xi,\xi\rangle\psi'\left(\frac{\vert\xi\vert^2}{k^{2\alpha}}\right)\widehat{u}\Vert_{L^2(\mathbb{R}^n)} = 0,$$
and as a consequence of the Plancherel theorem, since $\psi$ is bounded,
\begin{equation}\label{18062018E9}
	\lim_{k\rightarrow+\infty}\Vert\psi\left(\frac{\vert x\vert^2}{k^2}\right)\mathscr{F}^{-1}\left(\frac{2}{k^{2\alpha}}\langle B^T\xi,\xi\rangle\psi'\left(\frac{\vert\xi\vert^2}{k^{2\alpha}}\right)\widehat{u}\right)\Vert_{L^2(\mathbb{R}^n)} = 0.
\end{equation}
Finally, we derive from \eqref{18062018E10}, \eqref{18062018E2} and \eqref{18062018E9} that
\begin{gather}\label{09052018E6}
	\lim_{k\rightarrow+\infty}[\langle Bx,\nabla_x\rangle,\psi_k(x,D_x)]u = 0\quad \text{in $L^2(\mathbb{R}^n)$}.
\end{gather}
\textbf{2.} Now, we prove that
\begin{equation}\label{19062018E3}
	\lim_{k\rightarrow+\infty}[\Tr^s(-Q\nabla^2_x),\psi_k(x,D_x)]u = 0\quad \text{in $L^2(\mathbb{R}^n)$}.
\end{equation}
Since Fourier multipliers are commutative, we have
$$[\Tr^s(-Q\nabla^2_x),\psi_k(x,D_x)]u = [\Tr^s(-Q\nabla^2_x),\psi\left(\frac{\vert x\vert^2}{k^2}\right)]v_k,\quad \text{where}\quad v_k = \psi\left(\frac{D^2_x}{k^{2\alpha}}\right)u,$$
and it follows from the Plancherel theorem that
\begin{multline}\label{16032018E4}
	\Vert[\Tr^s(-Q\nabla^2_x),\psi_k(x,D_x)]u\Vert_{L^2(\mathbb{R}^n)} \\[5pt]
	= \frac{1}{(2\pi)^{\frac{n}{2}}} \Vert\vert Q^{\frac{1}{2}}\xi\vert^{2s} \big(k^n\widehat{\varphi}(k\xi)\ast\widehat{v_k}\big) - k^n\widehat{\varphi}(k\xi)\ast\big(\vert Q^{\frac{1}{2}}\xi\vert^{2s}\widehat{v_k}\big)\Vert_{L^2(\mathbb{R}^n)},
\end{multline}
where
$$\varphi(x) = \psi(\vert x\vert^2),\quad x\in\mathbb{R}^n.$$
Moreover, we have that for all $\xi\in\mathbb{R}^n$,
\begin{multline}\label{16032018E5}
	\vert Q^{\frac{1}{2}}\xi\vert^{2s}\big(k^n\widehat{\varphi}(k\xi)\ast\widehat{v_k}\big) - k^n\widehat{\varphi}(k\xi)\ast\big(\vert Q^{\frac{1}{2}}\xi\vert^{2s}\widehat{v_k}\big) \\
	= \int_{\mathbb{R}^n}k^n\big(\vert Q^{\frac{1}{2}}\xi\vert^{2s}-\vert Q^{\frac{1}{2}}\eta\vert^{2s}\big)\widehat{\varphi}(k(\xi-\eta))\widehat{v_k}(\eta)d\eta.
\end{multline}
When $2s>1$, we use Lemma \ref{16032018L1}, which yields that there exists a positive constant $c>0$ such that
$$\forall\xi,\eta\in\mathbb{R}^n,\quad \left\vert\vert Q^{\frac{1}{2}}\xi\vert^{2s}-\vert Q^{\frac{1}{2}}\eta\vert^{2s}\right\vert\le c\left(\vert\xi-\eta\vert^{2s} + \vert\eta\vert^{2s-1}\vert\xi-\eta\vert\right),$$
to derive from \eqref{16032018E4} and \eqref{16032018E5} that
\begin{multline}\label{16032018E3}
	\Vert[\Tr^s(-Q\nabla^2_x),\psi_k(x,D_x)]u\Vert_{L^2(\mathbb{R}^n)}\le 
	c\Vert k^n\big(\vert\xi\vert^{2s}\vert\widehat{\varphi}(k\xi)\vert\big)\ast\vert\widehat{v_k}\vert\Vert_{L^2(\mathbb{R}^n)} \\[5pt]
	+ c\Vert k^n\big(\vert\xi\vert\vert\widehat{\varphi}(k\xi)\vert\big)\ast\left(\vert\xi\vert^{2s-1}\vert\widehat{v_k}\vert\right)\Vert_{L^2(\mathbb{R}^n)}.
\end{multline}
Yet, as a consequence of Young's inequality and a change of variable, we first get that
\begin{align}\label{16032018E1}
	\Vert k^n\big(\vert\xi\vert^{2s}\vert\widehat{\varphi}(k\xi)\vert\big)\ast\vert\widehat{v_k}\vert\Vert_{L^2(\mathbb{R}^n)}
	& \le \Vert k^n\vert\xi\vert^{2s}\widehat{\varphi}(k\xi)\Vert_{L^1(\mathbb{R}^n)}\Vert\widehat{v_k}\Vert_{L^2(\mathbb{R}^n)} \\[5pt]
	& \le k^{-2s}\Vert\vert\xi\vert^{2s}\widehat{\varphi}(\xi)\Vert_{L^1(\mathbb{R}^n)}\Vert\widehat{u}\Vert_{L^2(\mathbb{R}^n)}\underset{k\rightarrow+\infty}{\longrightarrow}0. \nonumber
\end{align}
It follows from the very same arguments that
\begin{align}\label{16032018E2}
	&\ \Vert k^n\big(\vert\xi\vert\vert\widehat{\varphi}(k\xi)\vert\big)\ast\big(\vert\xi\vert^{2s-1}\vert\widehat{v_k}\vert\big)\Vert_{L^2(\mathbb{R}^n)} \\[5pt]
	\le &\ \Vert k^n\vert\xi\vert \widehat{\varphi}(k\xi)\Vert_{L^1(\mathbb{R}^n)}\Vert\vert\xi\vert^{2s-1}\varphi\left(k^{-\alpha}\xi\right)\widehat{u}\Vert_{L^2(\mathbb{R}^n)}\nonumber \\[5pt]
	\le &\ k^{-1+(2s-1)\alpha}\Vert\vert\xi\vert\widehat{\varphi}(\xi)\Vert_{L^1(\mathbb{R}^n)}
	\Vert\vert\xi\vert^{2s-1}\varphi(\xi)\Vert_{L^{\infty}(\mathbb{R}^n)}\Vert\widehat{u}\Vert_{L^2(\mathbb{R}^n)}\underset{k\rightarrow+\infty}{\longrightarrow}0, \nonumber
\end{align}
since $(2s-1)\alpha<1$. Then, \eqref{19062018E3} follows from (\ref{16032018E3}), (\ref{16032018E1}) and (\ref{16032018E2}). \\[5pt]
When $0<2s\le 1$, Lemma \ref{16032018L1} yields that there exists a positive constant $c>0$ such that
\begin{gather}\label{16032018E6}
	\forall\xi,\eta\in\mathbb{R}^n,\quad \left\vert\vert Q^{\frac{1}{2}}\xi\vert^{2s}-\vert Q^{\frac{1}{2}}\eta\vert^{2s}\right\vert\le c\vert\xi-\eta\vert^{2s}.
\end{gather}
Thus, it follows from \eqref{16032018E4}, \eqref{16032018E5}, \eqref{16032018E1} and \eqref{16032018E6} that
$$\Vert[\Tr^s(-Q\nabla^2_x),\psi_k(x,D_x)]u\Vert_{L^2(\mathbb{R}^n)}\le c\Vert k^n\big(\vert\xi\vert^{2s}\vert\widehat{\varphi}(k\xi)\vert\big)\ast\vert\widehat{v_k}\vert\Vert_{L^2(\mathbb{R}^n)}
\underset{k\rightarrow+\infty}{\longrightarrow}0,$$
and \eqref{19062018E3} is proved in this case. Therefore, we derive \eqref{09052018E9} from \eqref{09052018E6} and \eqref{19062018E3}, which ends the proof of Proposition \ref{15032018P3}.
\end{proof}

Thanks to Proposition \ref{15032018P3}, we can compute explicitly the adjoints of fractional Ornstein-Uhlenbeck operators :

\begin{cor}\label{16032018C2} The adjoint of the the fractional Ornstein-Uhlenbeck operator $\mathcal{P}$ defined in \eqref{06022018E3} and equipped with the domain \eqref{19062018E1} is given by
$$\mathcal{P}^* = \frac{1}{2}\Tr^s(-Q\nabla^2_x) -\langle Bx,\nabla_x\rangle - \Tr(B),$$
with domain 
$$D(\mathcal{P}^*) = \{u\in L^2(\mathbb{R}^n),\quad \mathcal{P}^*u\in L^2(\mathbb{R}^n)\}.$$
\end{cor}

\begin{proof} Let $\mathcal{Q}$ be the pseudodifferential operator defined by
$$\mathcal{Q} = \frac{1}{2}\Tr^s(-Q\nabla^2_x) - \langle Bx,\nabla_x\rangle - \Tr(B),$$
and equipped with the domain
$$D(\mathcal{Q}) = \{u\in L^2(\mathbb{R}^n),\quad \mathcal{Q}u\in L^2(\mathbb{R}^n)\}.$$
Let $u\in D(\mathcal{P})$ and $v\in D(\mathcal{Q})$. From Proposition \ref{15032018P3} applied respectively to the operators
$$\frac{1}{2}\Tr^s(-Q\nabla^2_x) + \langle Bx,\nabla_x\rangle\quad\text{and}\quad \frac{1}{2}\Tr^s(-Q\nabla^2_x) - \langle Bx,\nabla_x\rangle,$$
there exist some sequences $(u_k)_k$ and $(v_k)_k$ of $\mathscr{S}(\mathbb{R}^n)$ such that
$$\lim_{k\rightarrow+\infty} u_k = u,\quad \lim_{k\rightarrow+\infty}\mathcal{P}u_k = \mathcal{P} u\quad \text{in $L^2(\mathbb{R}^n)$,}$$
and
$$\lim_{k\rightarrow+\infty} v_k = v,\quad \lim_{k\rightarrow+\infty}\mathcal{Q}v_k = \mathcal{Q} v\quad \text{in $L^2(\mathbb{R}^n)$.}$$
Yet, it follows from an integration by parts that
$$\forall k\geq0,\quad \langle \mathcal{P}u_k,v_k\rangle_{L^2(\mathbb{R}^n)} = \langle u_k,\mathcal{Q}v_k\rangle_{L^2(\mathbb{R}^n)},$$
and passing to the limit, we deduce that
$$\langle \mathcal{P}u,v\rangle_{L^2(\mathbb{R}^n)} = \langle u,\mathcal{Q}v\rangle_{L^2(\mathbb{R}^n)}.$$
This equality shows that $D(\mathcal{Q})\subset D(\mathcal{P}^*)$ and $\mathcal{P}^*v = \mathcal{Q}v$ for all $v\in D(\mathcal{Q})$.  
Conversely, if $v\in D(\mathcal{P}^*)$, we get that for all $u\in\mathscr{S}(\mathbb{R}^n)$,
$$\langle\mathcal{Q}v,u\rangle_{\mathscr{S}'(\mathbb{R}^n),\mathscr{S}(\mathbb{R}^n)} 
= \langle v,\overline{\mathcal{P}u}\rangle_{L^2(\mathbb{R}^n)}
= \langle v,\mathcal{P}\overline{u}\rangle_{L^2(\mathbb{R}^n)}
= \langle\mathcal P^*v,\overline{u}\rangle_{L^2(\mathbb{R}^n)}
= \langle\mathcal{P}^*v,u\rangle_{\mathscr{S}'(\mathbb{R}^n),\mathscr{S}(\mathbb{R}^n)},$$
where $\langle\cdot,\cdot\rangle_{\mathscr{S}'(\mathbb{R}^n),\mathscr{S}(\mathbb{R}^n)}$ stands for the duality bracket of $\mathscr{S}'(\mathbb{R}^n)$ and $\mathscr{S}(\mathbb{R}^n)$, which proves that $\mathcal{Q}v = \mathcal{P}^*v\in L^2(\mathbb{R})$ and $D(\mathcal{P}^*)\subset D(\mathcal{Q})$.
\end{proof}

Another consequence of Proposition \ref{15032018P3} is the positivity property of fractional Ornstein-Uhlenbeck operators up to a constant :

\begin{cor}\label{16032018C1} Let $\mathcal{P}$ be the fractional Ornstein-Uhlenbeck operator defined in \eqref{06022018E3} and equipped with the domain \eqref{19062018E1}. Then, we have that for all $u\in D(\mathcal{P})$,
$$\Reelle\langle\mathcal{P}u,u\rangle_{L^2(\mathbb{R}^n)} + \frac{1}{2}\Tr(B)\Vert u\Vert^2_{L^2(\mathbb{R}^n)}\geq0.$$
\end{cor}

\begin{proof} Let $u\in D(\mathcal{P})$. From Proposition \ref{15032018P3}, there exists a sequence $(u_k)_k$ of $\mathscr{S}(\mathbb{R}^n)$ such that
$$\lim_{k\rightarrow+\infty} u_k = u,\quad \lim_{k\rightarrow+\infty}\mathcal{P}u_k = \mathcal{P} u\quad \text{in $L^2(\mathbb{R}^n)$.}$$
It follows from Corollary \ref{16032018C2} that for all $k\geq0$,
\begin{multline*}
	\langle\mathcal{P}u_k,u_k\rangle_{L^2(\mathbb{R}^n)} 
	= \frac{1}{2}\Vert\Tr^{\frac{s}{2}}(-Q\nabla^2_x)u_k\Vert^2_{L^2(\mathbb{R}^n)} - \langle u_k,\langle Bx,\nabla_x\rangle u_k\rangle_{L^2(\mathbb{R}^n)} - \Tr(B)\Vert u_k\Vert^2_{L^2(\mathbb{R}^n)} \\[5pt]
	= \Vert\Tr^{\frac{s}{2}}(-Q\nabla^2_x)u_k\Vert^2_{L^2(\mathbb{R}^n)} - \langle u_k,\mathcal{P}u_k\rangle_{L^2(\mathbb{R}^n)} - \Tr(B)\Vert u_k\Vert^2_{L^2(\mathbb{R}^n)}.
\end{multline*}
Therefore, we have that for all $k\geq0$,
$$\Reelle\langle\mathcal{P}u_k,u_k\rangle_{L^2(\mathbb{R}^n)} + \frac{1}{2}\Tr(B)\Vert u_k\Vert^2_{L^2(\mathbb{R}^n)}\geq0,$$
and Corollary \ref{16032018C1} follows passing to the limit.
\end{proof}

\subsection{Generated semigroup} By using some basics of the semigroup theory, we now prove that fractional Ornstein-Uhlenbeck operators generate strongly continuous semigroups. First, we need to check that the operators are densely defined and closed :

\begin{lem}\label{16032018L5} The fractional Ornstein-Uhlenbeck operator $\mathcal{P}$ defined in \eqref{06022018E3} and equipped with the domain \eqref{19062018E1} is densely defined and closed.
\end{lem}

\begin{proof} $\mathcal{P}$ is densely defined since $D(\mathcal{P})$ contains the Schwartz space $\mathscr{S}(\mathbb{R}^n)$. Now, we consider $u,v\in L^2(\mathbb{R}^n)$ and $(u_k)$ a sequence of $D(\mathcal{P})$ such that
$$\lim_{k\rightarrow+\infty}u_k = u\quad \text{and}\quad \lim_{k\rightarrow+\infty}\mathcal{P}u_k = v\quad \text{in $L^2(\mathbb{R}^n)$}.$$
We have that for all $\varphi\in\mathscr{S}(\mathbb{R}^n)$,
$$\langle\mathcal{P}u_k,\varphi\rangle_{\mathscr{S}'(\mathbb{R}^n),\mathscr{S}(\mathbb{R}^n)} 
= \langle u_k, \overline{\mathcal{P}^*\varphi}\rangle_{L^2(\mathbb{R}^n)} \\[5pt]
\underset{k\rightarrow+\infty}{\longrightarrow}\langle u,\overline{\mathcal{P}^*\varphi}\rangle_{L^2(\mathbb{R}^n)} 
= \langle\mathcal{P}u,\varphi\rangle_{\mathscr{S}'(\mathbb{R}^n),\mathscr{S}(\mathbb{R}^n)}.$$
On the other hand, the following convergence holds for all $\varphi\in\mathscr{S}(\mathbb{R}^n)$,
$$\lim_{k\rightarrow+\infty}\langle\mathcal{P}u_k,\varphi\rangle_{\mathscr{S}'(\mathbb{R}^n),\mathscr{S}(\mathbb{R}^n)} = \langle v,\varphi\rangle_{\mathscr{S}'(\mathbb{R}^n),\mathscr{S}(\mathbb{R}^n)},$$
and it implies that $v = \mathcal{P}u$. This shows that $\mathcal{P}$ is a closed operator.
\end{proof}

\begin{prop}\label{16032018P1} The fractional Ornstein-Uhlenbeck operator $\mathcal{P}$ defined in \eqref{06022018E3} and equipped with the domain \eqref{19062018E1} generates a strongly continuous semigroup $(e^{-t\mathcal{P}})_{t\geq0}$ on $L^2(\mathbb{R}^n)$ which satisfies that for all $t\geq0$ and $u\in L^2(\mathbb{R}^n)$,
\begin{gather}\label{09052018E10}
	\Vert e^{-t\mathcal{P}}u\Vert_{L^2(\mathbb{R}^n)}\le e^{\frac{1}{2}\Tr(B)t}\ \Vert u\Vert_{L^2(\mathbb{R}^n)}.
\end{gather}
\end{prop}

\begin{proof} We consider the operator
$$\mathcal{P}_{co} = \mathcal{P} + \frac{1}{2}\Tr(B) =  \frac{1}{2}\Tr^s(-Q\nabla^2_x) + \langle Bx,\nabla_x\rangle + \frac{1}{2}\Tr(B)$$
equipped with the domain $D(\mathcal{P})$. It follows from Corollary \ref{16032018C2} that the adjoint of $\mathcal{P}_{co}$ is given by 
$$(\mathcal{P}_{co})^* = \frac{1}{2}\Tr^s(-Q\nabla^2_x) -\langle Bx,\nabla_x\rangle - \frac{1}{2}\Tr(B),$$
and Corollary \ref{16032018C1} shows that both $\mathcal{P}_{co}$ and $(\mathcal{P}_{co})^*$ are accretive operators. Therefore, the existence of the strongly continuous contraction semigroup $(e^{-t\mathcal{P}_{co}})_{t\geq0}$ follows from the Lumer-Phillips theorem, see e.g. Chapter 1, Corollary 4.4 in \cite{MR710486}, since $\mathcal{P}_{co}$ is a densely defined closed operator from Lemma \ref{16032018L5}. As a consequence, $\mathcal{P}$ generates a strongly continuous semigroup $(e^{-t\mathcal{P}})_{t\geq0}$ which satisfies \eqref{09052018E10}.
\end{proof}

In the remaining of this subsection, we compute the Fourier transforms of semigroups generated by fractional Ornstein-Uhlenbeck operators. We begin with some formal manipulations to derive a formal expression of these Fourier transforms. Let $u = e^{-t\mathcal P}u_0$ be the mild solution of the equation
$$\left\{\begin{array}{l}
	\partial_tu(t,x) + \frac{1}{2}\Tr^s(-Q\nabla^2_x)u(t,x) + \langle Bx,\nabla_x\rangle u(t,x) = 0, \\[5pt]
	u(0,\cdot) = u_0\in L^2(\mathbb{R}^n).
\end{array}\right.$$
By passing to Fourier side, $\widehat{u}$ is the solution of the Cauchy problem
$$\left\{\begin{array}{l}
	\partial_t\widehat{u}(t,\xi) + \frac{1}{2}\vert Q^{\frac{1}{2}}\xi\vert^{2s}\widehat{u}(t,\xi) - \langle B^T\xi,\nabla_{\xi}\rangle \widehat{u}(t,\xi) - \Tr(B)\widehat{u}(t,\xi) = 0, \\[5pt]
	\widehat{u}(0,\cdot) = \widehat{u_0}.
\end{array}\right.$$
We consider the function $v$ implicitly defined by $\widehat{u}(t,\xi) = v(t,e^{tB^T}\xi)e^{\Tr(B)t}$. An immediate computation shows that $v$ satisfies
$$\left\{\begin{array}{l}
	\partial_tv(t,\eta) + \frac{1}{2}\vert Q^{\frac{1}{2}}e^{-tB^T}\eta\vert^{2s}v(t,\eta) = 0, \\[5pt]
	v(0,\cdot) = \widehat{u_0},
\end{array}\right.$$
and therefore,
$$v(t,\eta) = \exp\left[-\frac{1}{2}\int_0^t\vert Q^{\frac{1}{2}}e^{-\tau B^T}\eta\vert^{2s}d\tau\right]\widehat{u_0}(\eta).$$
Finally, we deduce that the Fourier transform of the function $u$ is given by
$$\widehat{u}(t,\xi) = e^{\Tr(B)t}\exp\left[-\frac{1}{2}\int_0^t\vert Q^{\frac{1}{2}}e^{\tau B^T}\xi\vert^{2s}d\tau\right]\widehat{u_0}(e^{tB^T}\xi).$$
We justify these informal calculations in the following lemma and proposition :

\begin{lem}\label{10052018L1} Let $s>0$, $B$ and $Q$ be real $n\times n$ matrices, where $Q$ is symmetric positive semidefinite. For all $t\geq0$, we consider the bounded operator $T(t):L^2(\mathbb{R}^n)\rightarrow L^2(\mathbb{R}^n)$ defined by
\begin{gather}\label{15032018E1}
	\widehat{T(t)u} = e^{\Tr(B)t}\exp\left[-\frac{1}{2}\int_0^t\vert Q^{\frac{1}{2}}e^{\tau B^T}\cdot\vert^{2s}d\tau\right]\widehat{u}(e^{tB^T}\cdot), \quad u\in L^2(\mathbb{R}^n).
\end{gather}
Then $(T(t))_{t\geq0}$ defines a strongly continuous semigroup on $L^2(\mathbb{R}^n)$ satisfying
\begin{equation}\label{03082018E3}
	\forall t\geq0, \forall u\in L^2(\mathbb{R}^n),\quad \Vert T(t)u\Vert_{L^2(\mathbb{R}^n)}\le e^{\frac{1}{2}\Tr(B)t}\Vert u\Vert_{L^2(\mathbb{R}^n)}.
\end{equation}
\end{lem}

\begin{proof} The fact that $(T(t))_{t\geq0}$ satisfies the semigroup property, that is 
$$\forall t,s\geq0,\quad T(t+s) = T(t)T(s),$$
follows from a direct computation. We check that it is strongly continuous, i.e.
\begin{gather}\label{14032018E3}
	\forall u\in L^2(\mathbb{R}^n),\quad \lim_{t\rightarrow0^+}\Vert T(t)u-u\Vert_{L^2(\mathbb{R}^n)} = 0.
\end{gather}
Let $u\in\mathscr{S}(\mathbb{R}^n)$. First, we have the following convergence :
\begin{gather}\label{10052018E1}
	\lim_{t\rightarrow0^+}\Vert \widehat{T(t)u}\Vert_{L^2(\mathbb{R}^n)} = \Vert \widehat{u}\Vert_{L^2(\mathbb{R}^n)}.
\end{gather}
Indeed, as a consequence of \eqref{15032018E1} and a change of variable, we get that
\begin{gather}\label{10052018E2}
	\Vert \widehat{T(t)u}\Vert_{L^2(\mathbb{R}^n)} = e^{\frac{1}{2}\Tr(B)t}
\Vert\exp\left[-\frac{1}{2}\int_0^t\vert Q^{\frac{1}{2}}e^{-\tau B^T}\cdot\vert^{2s}d\tau\right]\widehat{u}\Vert_{L^2(\mathbb{R}^n)}.
\end{gather}
Moreover, the following convergence stands almost everywhere on $\mathbb{R}^n$
$$e^{\frac{1}{2}\Tr(B)t}\exp\left[-\frac{1}{2}\int_0^t\vert Q^{\frac{1}{2}}e^{-\tau B^T}\cdot\vert^{2s}d\tau\right]\widehat{u}\underset{t\rightarrow0}{\longrightarrow}\widehat{u},$$
and the following domination holds
$$\exists c>0,\forall t\in[0,1],\quad \vert e^{\frac{1}{2}\Tr(B)t}\exp\left[-\frac{1}{2}\int_0^t\vert Q^{\frac{1}{2}}e^{-\tau B^T}\cdot\vert^{2s}d\tau\right]\widehat{u}\vert\le c \vert\widehat{u}\vert.$$
Therefore, \eqref{10052018E1} is a consequence of the dominated convergence theorem.
Moreover, since $\widehat{u}$ is a continuous function, we have that for almost all $\xi\in\mathbb{R}^n$,
\begin{gather}\label{23052018E2}
	\lim_{t\rightarrow0^+}\widehat{T(t)u}(\xi) = \widehat{u}(\xi).
\end{gather}
Thus, by applying a classical lemma of measure theory (see Lemma \ref{12022018L1} in appendix) and the Plancherel theorem, we get
\begin{gather}\label{23052018E3}
	\lim_{t\rightarrow0^+}\Vert T(t)u-u\Vert_{L^2(\mathbb{R}^n)} = 0.
\end{gather}
When $u\in L^2(\mathbb{R}^n)$, we consider $(u_k)_k$ a sequence of $\mathscr{S}(\mathbb{R}^n)$ converging to $u$ in $L^2(\mathbb{R}^n)$. It follows from \eqref{10052018E2} and the Plancherel theorem that for all $t\geq0$,
\begin{align*}
	\Vert T(t)u-u\Vert_{L^2(\mathbb{R}^n)} & \le \Vert T(t)u-T(t)u_k\Vert_{L^2(\mathbb{R}^n)} + \Vert T(t)u_k - u_k\Vert_{L^2(\mathbb{R}^n)} + \Vert u_k-u\Vert_{L^2(\mathbb{R}^n)} \\[5pt]
	& \le e^{\frac{1}{2}\Tr(B)t}\Vert u_k-u\Vert_{L^2(\mathbb{R}^n)} + \Vert T(t)u_k - u_k\Vert_{L^2(\mathbb{R}^n)} + \Vert u_k-u\Vert_{L^2(\mathbb{R}^n)}.
\end{align*}
Thus, it follows from \eqref{23052018E3} that
$$\limsup_{t\rightarrow0^+}\Vert T(t)u-u\Vert_{L^2(\mathbb{R}^n)}\le 2\Vert u_k-u\Vert_{L^2(\mathbb{R}^n)}\underset{k\rightarrow+\infty}{\longrightarrow}0,$$
and \eqref{14032018E3} is proved. Finally, \eqref{03082018E3} is a straightforward consequence of \eqref{10052018E2} and the Plancherel theorem. This ends the proof of Lemma \ref{10052018L1}.
\end{proof}

\begin{prop}\label{16032018P2} Let $\mathcal{P}$ be the fractional Ornstein-Uhlenbeck operator $\mathcal{P}$ defined in \eqref{06022018E3} and equipped with the domain \eqref{19062018E1}. Then, we have that for all $t\geq0$ and $u\in L^2(\mathbb{R}^n)$,
$$\widehat{e^{-t\mathcal{P}}u} = e^{\Tr(B)t}\exp\left[-\frac{1}{2}\int_0^t\vert Q^{\frac{1}{2}}e^{\tau B^T}\cdot\vert^{2s}d\tau\right]\widehat{u}(e^{tB^T}\cdot).$$
\end{prop}

\begin{proof} We consider $(T(t))_{t\geq0}$ the strongly continuous semigroup defined on $L^2(\mathbb{R}^n)$ by \eqref{15032018E1} and $(A,D(A))$ its infinitesimal generator. It is sufficient to prove that $A = -\mathcal{P}$ to end the proof of Proposition \ref{16032018P2} since it implies that $e^{-t\mathcal{P}}u = T(t)u$ for all $t\geq0$ and $u\in L^2(\mathbb{R}^n)$. \\[5pt]
\textbf{1.} We first check that $\mathscr{S}(\mathbb{R}^n)\subset D(A)$ and 
\begin{equation}\label{03082018E1}
	\forall u\in\mathscr{S}(\mathbb{R}^n),\quad Au = -\mathcal{P}u.
\end{equation}
Let $u\in\mathscr{S}(\mathbb{R}^n)$. It follows from the mean value theorem that
\begin{gather}\label{15032018E2}
	\forall \xi\in\mathbb{R}^n, \forall t\in[0,1],\quad \vert \widehat{T(t)u}(\xi) - \widehat{u}(\xi)\vert\le \sup_{\tau\in[0,1]}\vert\partial_t\!\widehat{T(t)u}(\xi)_{\vert_{t=\tau}}\vert\ \vert t\vert.
\end{gather}
Yet, we get from \eqref{15032018E1} that for all $\tau\in[0,1]$ and $\xi\in\mathbb{R}^n$,
\begin{multline}\label{15032018E3}
	\partial_t\!\widehat{T(t)u}(\xi)_{\vert_{t=\tau}} = e^{\Tr(B)\tau}\exp\left[-\frac{1}{2}\int_0^{\tau}\vert Q^{\frac{1}{2}}e^{\tau' B^T}\xi\vert^{2s}d\tau'\right] \\
	\left(- \frac{1}{2}\vert Q^{\frac{1}{2}}\cdot\vert^{2s}\widehat{u} + \langle B^T\cdot,\nabla_{\xi}\widehat{u}\rangle+\Tr(B)\widehat{u}\right)(e^{\tau B^T}\xi),
\end{multline}
and as
$$\widehat{\mathcal{P}u} = \frac{1}{2}\vert Q^{\frac{1}{2}}\cdot\vert^{2s}\widehat{u} -\langle B^T\cdot,\nabla_{\xi}\widehat{u}\rangle-\Tr(B)\widehat{u},$$
we have that
\begin{gather}\label{10052018E3}
	\partial_t\!\widehat{T(t)u}(\xi)_{\vert_{t=\tau}} = -\widehat{T(\tau)\mathcal{P}u}(\xi).
\end{gather}
Since $\widehat{u}\in\mathscr{S}(\mathbb{R}^n)$, it follows from \eqref{15032018E3} that
$$\exists c_u>0, \forall\xi\in\mathbb{R}^n, \forall\tau\in[0,1],\quad \vert\partial_t\!\widehat{T(t)u}(\xi)_{\vert_{t=\tau}}\vert\le\frac{c_u}{1+\vert\xi\vert^n}.$$
Combining this estimation with \eqref{15032018E2}, we get that
\begin{gather}\label{20032018E1}
	\forall \xi\in\mathbb{R}^n, \forall t\in(0,1],\quad \frac{1}{t}\vert\widehat{T(t)u}(\xi)-\widehat{u}(\xi)\vert\le\frac{c_u}{1+\vert\xi\vert^n}.
\end{gather}
Moreover, we deduce from \eqref{10052018E3} that
$$\forall\xi\in\mathbb{R}^n,\quad \partial_t\!\widehat{T(t)u}(\xi)_{\vert_{t=0}}= -\widehat{\mathcal{P}u}(\xi),$$
and this equality can be written as
\begin{gather}\label{15032018E6}
	\forall\xi\in\mathbb{R}^n,\quad \lim_{t\rightarrow0^+}\left[\frac{1}{t}(\widehat{T(t)u}(\xi)-\widehat{u}(\xi))\right] = -\widehat{\mathcal{P}u}(\xi).
\end{gather}
As a consequence of (\ref{20032018E1}), (\ref{15032018E6}) and the dominated convergence theorem, it follows that
\begin{gather}\label{15032018E5}
	\lim_{t\rightarrow0^+}\left[\frac{1}{t}\Vert\widehat{T(t)u}-\widehat{u}\Vert_{L^2(\mathbb{R}^n)}\right] = \Vert \widehat{\mathcal{P}u}\Vert_{L^2(\mathbb{R}^n)}.
\end{gather}
We deduce from (\ref{15032018E6}), (\ref{15032018E5}), Lemma \ref{12022018L1} and the Plancherel theorem that
$$\lim_{t\rightarrow0^+}\left[\frac{1}{t}(T(t)u-u)\right] = -\mathcal{P}u\quad \text{in $L^2(\mathbb{R}^n)$}.$$
Therefore, $u\in D(A)$ and $Au = -\mathcal{P}u$. This proves that \eqref{03082018E1} holds. \\[5pt]
\textbf{2.} The second step consists in proving that $(-\mathcal{P})\subset A$, that is, $D(\mathcal P)\subset D(A)$ and $Au = -\mathcal Pu$ for all $u\in D(\mathcal P)$. Let $u\in D(\mathcal{P})$. It follows from Proposition~\ref{15032018P3} that there exists $(u_k)_k$ a sequence of Schwartz functions satisfying
\begin{equation}\label{03082018E2}
	\lim_{k\rightarrow+\infty} u_k = u,\quad \lim_{k\rightarrow+\infty}\mathcal{P}u_k = \mathcal{P} u\quad \text{in $L^2(\mathbb{R}^n)$.}
\end{equation}
We deduce from \eqref{03082018E1} that $Au_k = -\mathcal{P}u_k$ for all $k\geq0$ since $u_k\in\mathscr{S}(\mathbb{R}^n)$ and \eqref{03082018E2} implies the following convergence
$$\lim_{k\rightarrow+\infty}(u_k,Au_k) = (u,-\mathcal{P}u)\quad \text{in $L^2(\mathbb{R}^n)\times L^2(\mathbb{R}^n)$.}$$
It follows from the classical corollary of the Hille-Yosida theorem that $A$ is a closed operator, see e.g. \cite{MR710486} (Chapter 1, Corollary 3.8). Therefore, $u\in D(A)$ and $Au = -\mathcal Pu\in L^2(\mathbb{R}^n)$. We proved that $(-\mathcal{P})\subset A$. \\[5pt]
\textbf{3.} Finally, we check that $A\subset (-\mathcal{P})$, that is, $D(A)\subset D(\mathcal P)$ and $-\mathcal Pu=Au$ for all $u\in D(A)$. Since both operators $-\mathcal{P}$ and $A$ are infinitesimal generators of strongly continuous semigroups satisfying from Proposition \ref{16032018P1} and Lemma \ref{10052018L1} that for all $t\geq0$ and $u\in L^2(\mathbb{R}^n)$,
$$\Vert e^{-t\mathcal{P}}u\Vert_{L^2(\mathbb{R}^n)}\le e^{\frac{1}{2}\Tr(B)t}\Vert u\Vert_{L^2(\mathbb{R}^n)}\quad \text{and}\quad 
\Vert e^{tA}u\Vert_{L^2(\mathbb{R}^n)}\le e^{\frac{1}{2}\Tr(B)t}\Vert u\Vert_{L^2(\mathbb{R}^n)},$$
it follows from \cite{MR710486} (Chapter 1, Corollary 3.8) that there exists a real number $\mu>\frac{1}{2}\Tr(B)$  such that the linear operators $-\mathcal{P}-\mu:D(\mathcal{P})\rightarrow L^2(\mathbb{R}^n)$ and $A-\mu:D(A)\rightarrow L^2(\mathbb{R}^n)$ are bijective. Let $u\in D(A)$ and $v = (A-\mu)u\in L^2(\mathbb{R}^n)$. Since the operator $-\mathcal{P}-\mu$ is bijective, there exists a unique $w\in D(\mathcal{P})$ such that $v = (-\mathcal{P}-\mu)w$. By using that $(-\mathcal{P})\subset A$, we deduce that $v = (A-\mu)w$. Since $A-\mu$ is injective and $v=(A-\mu)u=(A-\mu)w$, we get that $u = w\in D(\mathcal{P})$. This implies that $A\subset (-\mathcal{P})$ and then $A = -\mathcal{P}$.
\end{proof}

Theorem \ref{06022018T1} is now a consequence of Propositions \ref{16032018P1} and \ref{16032018P2}.

\section{Gevrey regularizing effects of fractional Ornstein-Uhlenbeck semigroups}
\label{sec_GreofOUs}

In this section, we prove Theorem \ref{30082018T1}. Let $\mathcal{P}$ be the fractional Ornstein-Uhlenbeck operator defined in \eqref{06022018E3} and equipped with the domain \eqref{19062018E1}. We assume that the Kalman rank condition \eqref{10052018E4} holds and we denote by $0\le r\le n-1$ the smallest integer satisfying \eqref{05052018E4}. Moreover, for all $0\le k\le r$, we consider $\mathbb P_k$ the orthogonal projection onto the vector subspace $V_k$ defined in \eqref{01062018E2}.

\subsection{First estimates} The next proposition states that the semigroup $(e^{-t\mathcal{P}})_{t\geq0}$ is smoothing in the Gevrey type space $G^{\frac1{2s}}(\mathbb R^n)$ defined in \eqref{31072018E1}, but only provides rough controls of the associated seminorms :

\begin{prop}\label{15022018P2} There exists a positive constant $C>1$ such that for all $k\in\{0,\ldots,r\}$, $q>0$, $0<t<1$ and $u\in L^2(\mathbb{R}^n)$,
$$\Vert\vert\mathbb{P}_kD_x\vert^q e^{-t\mathcal{P}}u\Vert_{L^2(\mathbb{R}^n)}
\le C^{1+q}\left[\frac{M^s_t}{t^{\frac{1}{2}+k}}\right]^q\ e^{\frac{1}{2}\Tr(B)t}\ q^{\frac{q}{2s}}\ 
\Vert u\Vert_{L^2(\mathbb{R}^n)},$$
where
$$M^s_t = \sup_{\xi\in\mathbb{S}^{n-1}}\left[\int_0^t\vert Q^{\frac{1}{2}}e^{\tau B^T}\xi\vert^2d\tau\right]^{\frac{1}{2}}\left[\int_0^t\vert Q^{\frac{1}{2}}e^{\tau B^T}\xi\vert^{2s}d\tau\right]^{-\frac{1}{2s}}.$$
\end{prop}

We check in Lemma \ref{19032018L1} further in this section that $M^s_t$ is well-defined and satisfies $0<M^s_t<+\infty$. Moreover, the study of the asymptotics for small times of the term $M^s_t$ appearing in the above statement is also postponed further in this section.

\begin{proof} The key ingredients of this proof are on the one hand the explicit formula for the Fourier transform of the evolution operators $e^{-t\mathcal{P}}$ derived in Theorem \ref{06022018T1} and on the other hand some properties of the orthogonal projections $\Pi_0,\ldots,\Pi_r$ defined by 
\begin{enumerate}
\item[\textbf{1.}] $\Pi_0$ the orthogonal projection onto $V_0$, \\[-9pt]
\item[\textbf{2.}] $\Pi_{k+1}$ the orthogonal projection onto $W_k$ with $V_{k+1} = V_k \overset{\perp}{\oplus} W_k$, for all $0\le k\le r-1$,
\end{enumerate}
where the orthogonality is taken with respect to the canonical Euclidean structure, obtained by A. Lunardi in \cite{MR1475774}. We begin by deriving from the Plancherel theorem and Theorem \ref{06022018T1} that for all $k\in\{0,\ldots,r\}$, $q\geq0$, $0<t<1$ and $u\in \mathscr{S}(\mathbb{R}^n)$,
$$\Vert\vert\Pi_kD_x\vert^qe^{-t\mathcal{P}}u\Vert_{L^2(\mathbb{R}^n)}
\le \frac{e^{\Tr(B)t}}{(2\pi)^{\frac{n}{2}}}\Vert\vert\Pi_k\xi\vert^q \exp\left[-\frac{1}{2}\int_0^t\vert Q^{\frac{1}{2}}e^{\tau B^T}\xi\vert^{2s}\ d\tau\right]\widehat{u}(e^{tB^T}\xi)\Vert_{L^2(\mathbb{R}^n)}.$$
It follows from Lemma 3.1 in \cite{MR1475774} that there exists a positive constant $c>0$ only depending on $B$ and $Q$ such that
$$\forall k\in\{0,\ldots,r\}, \forall t\in(0,1),\quad \Vert\Pi_k e^{-tB^T}(Q_t)^{-\frac12}\Vert\le\frac{c}{t^{\frac{1}{2}+k}},$$
the notation $\Vert\cdot\Vert$ standing for the matrix norm on $M_n(\mathbb R)$ induced by the canonical Euclidean norm $\vert\cdot\vert$ on $\mathbb R^n$ and where the symmetric positive semidefinite matrices $Q_t$ are defined in \eqref{27082018E2}. We recall from the introduction that the non-degeneracy of the matrices $Q_t$ is implied by the Kalman rank condition. Therefore, we have that for all $k\in\{0,\ldots,r\}$, $q>0$, $0<t<1$ and $u\in \mathscr{S}(\mathbb{R}^n)$,
\begin{multline}\label{05022018E9}
	\Vert\vert\Pi_kD_x\vert^qe^{-t\mathcal{P}}u\Vert_{L^2(\mathbb{R}^n)}
	\le\left[\frac{c}{t^{\frac{1}{2}+k}}\right]^q\ \frac{e^{\Tr(B)t}}{(2\pi)^{\frac{n}{2}}} \\
	\Vert\vert (Q_t)^{\frac{1}{2}}e^{tB^T}\xi\vert^q\exp\left[-\frac{1}{2}\int_0^t\vert Q^{\frac{1}{2}}e^{\tau B^T}\xi\vert^{2s}\ d\tau\right]\widehat{u}(e^{tB^T}\xi)\Vert_{L^2(\mathbb{R}^n)}.
\end{multline}
Note that from the definition \eqref{27082018E2} of $Q_t$,
\begin{gather}\label{10052018E7}
	\forall t\in(0,1),\forall\xi\in\mathbb{R}^n,\quad\vert (Q_t)^{\frac{1}{2}}e^{tB^T}\xi\vert^2 = \langle Q_te^{tB^T}\xi,e^{tB^T}\xi\rangle = \int_0^t\vert Q^{\frac{1}{2}}e^{\tau B^T}\xi\vert^2d\tau.
\end{gather}
Let $\xi\in\mathbb{R}^n\setminus\{0\}$, and $(\rho,\sigma)$ be the polar coordinates of $\xi$, i.e. $\xi = \rho\sigma$ with $\rho>0$ and $\sigma\in\mathbb{S}^{n-1}$. Then, it follows from \eqref{10052018E7} and the estimate
$$\forall q>0, \forall x\geq0,\quad x^qe^{-x^{2s}} \le \left[\frac{q}{2es}\right]^{\frac{q}{2s}},$$
that for all $0<t<1$,
\begin{align}\label{10052018E8}
	&\ \vert (Q_t)^{\frac{1}{2}}e^{tB^T}\xi\vert^q\exp\left[-\frac{1}{2}\int_0^t\vert Q^{\frac{1}{2}}e^{\tau B^T}\xi\vert^{2s}\ d\tau\right] \\[5pt]
	= &\ \left[\int_0^t\vert Q^{\frac{1}{2}}e^{\tau B^T}\sigma\vert^2d\tau\right]^{\frac{q}{2}}\rho^q\exp\left[-\frac{1}{2}\left(\int_0^t\vert Q^{\frac{1}{2}}e^{\tau B^T}\sigma\vert^{2s}\ d\tau\right)\rho^{2s}\right] \nonumber \\[5pt]
	\le &\ \left[\int_0^t\vert Q^{\frac{1}{2}}e^{\tau B^T}\sigma\vert^2d\tau\right]^{\frac{q}{2}}\left[\frac{1}{2}\int_0^t\vert Q^{\frac{1}{2}}e^{\tau B^T}\sigma\vert^{2s}d\tau\right]^{-\frac{q}{2s}}\left[\frac{q}{2es}\right]^{\frac{q}{2s}}
	\le (M^s_t)^q\left[\frac{q}{es}\right]^{\frac{q}{2s}}, \nonumber
\end{align}
where
$$M^s_t = \sup_{\eta\in\mathbb{S}^{n-1}}\left[\int_0^t\vert Q^{\frac{1}{2}}e^{\tau B^T}\eta\vert^2d\tau\right]^{\frac{1}{2}}\left[\int_0^t\vert Q^{\frac{1}{2}}e^{\tau B^T}\eta\vert^{2s}d\tau\right]^{-\frac{1}{2s}}.$$
Therefore, we deduce from \eqref{05022018E9}, \eqref{10052018E8}, a change of variables and the Plancherel theorem that for all $k\in\{0,\ldots,r\}$, $q>0$, $0<t<1$ and $u\in \mathscr{S}(\mathbb{R}^n)$,
\begin{equation}\label{03092018E1}
	\Vert\vert\Pi_kD_x\vert^qe^{-t\mathcal{P}}u\Vert_{L^2(\mathbb{R}^n)}\le \left[\frac{cM^s_t}{t^{\frac{1}{2}+k}}\right]^q\
e^{\frac{1}{2}\Tr(B)t}\ \left[\frac{q}{es}\right]^{\frac{q}{2s}}\ \Vert u\Vert_{L^2(\mathbb{R}^n)}.
\end{equation}
By using that for all $k\in\{0,\ldots,r\}$, $\mathbb P_k = \Pi_0+\ldots+\Pi_k$, we obtain from Lemma \ref{05022018L2} and the Plancherel theorem that for all $k\in\{0,\ldots,r\}$, $q>0$, $0<t<1$ and $u\in \mathscr S(\mathbb{R}^n)$,
\begin{multline}\label{13022018E9}
	\Vert\vert \mathbb P_kD_x\vert^qe^{-t\mathcal{P}}u\Vert_{L^2(\mathbb{R}^n)} 
	= \Vert\vert (\Pi_0+\ldots+\Pi_k) D_x\vert^qe^{-t\mathcal{P}}u\Vert_{L^2(\mathbb{R}^n)} \\
	\le (r+1)^{\left(q-1\right)_+}\sum_{j=0}^k\Vert\vert\Pi_jD_x\vert^qe^{-t\mathcal{P}}u\Vert_{L^2(\mathbb{R}^n)}.
\end{multline}
We then deduce from \eqref{03092018E1} and \eqref{13022018E9} that for all $k\in\{0,\ldots,r\}$, $q>0$, $0<t<1$ and $u\in \mathscr{S}(\mathbb{R}^n)$,
\begin{align*}
	\Vert\vert \mathbb P_kD_x\vert^qe^{-t\mathcal{P}}u\Vert_{L^2(\mathbb{R}^n)}
	& \le (r+1)^{\left(q-1\right)_+}\sum_{j=0}^k\left[\frac{cM^s_t}{t^{\frac{1}{2}+j}}\right]^q e^{\frac{1}{2}\Tr(B)t}\ q^{\frac{q}{2s}}\ \Vert u\Vert_{L^2(\mathbb{R}^n)} \\
	& \le (r+1)^{\left(q-1\right)_+}\sum_{j=0}^k\left[\frac{cM^s_t}{t^{\frac{1}{2}+k}}\right]^q e^{\frac{1}{2}\Tr(B)t}\ q^{\frac{q}{2s}}\ \Vert u\Vert_{L^2(\mathbb{R}^n)} \\
	& = (r+1)^{1+\left(q-1\right)_+}\left[\frac{cM^s_t}{t^{\frac{1}{2}+k}}\right]^q e^{\frac{1}{2}\Tr(B)t}\ q^{\frac{q}{2s}}\ \Vert u\Vert_{L^2(\mathbb{R}^n)}.
\end{align*}
The previous inequality can be extended to all $u\in L^2(\mathbb R^n)$ since the Schwartz space $\mathscr{S}(\mathbb R^n)$ is dense in $L^2(\mathbb R^n)$. This ends the proof of Proposition \ref{15022018P2}.
\end{proof}

\subsection{Study of the term $M^s_t$} In order to obtain sharp asymptotics of the seminorms 
$\Vert\vert \mathbb P_kD_x\vert^{q} e^{-t\mathcal{P}}u\Vert_{L^2(\mathbb{R}^n)}$ as $t\rightarrow0^+$, we need to study the term
\begin{equation}\label{06052018E1}
	M^s_t = \sup_{\xi\in\mathbb{S}^{n-1}}\left[\int_0^t\vert Q^{\frac{1}{2}}e^{\tau B^T}\xi\vert^2d\tau\right]^{\frac{1}{2}}\left[\int_0^t\vert Q^{\frac{1}{2}}e^{\tau B^T}\xi\vert^{2s}d\tau\right]^{-\frac{1}{2s}}.
\end{equation}

\begin{lem}\label{19032018L1} For all $s>0$ and $t>0$, $M^s_t$ is well-defined and $0<M^s_t<+\infty$.
\end{lem}

\begin{proof} Let $\alpha\in\{2,2s\}$. We first check that
\begin{equation}\label{31072018E2}
	\forall t>0, \forall\xi\in\mathbb{S}^{n-1},\quad \int_0^t\vert Q^{\frac{1}{2}}e^{\tau B^T}\xi\vert^{\alpha}d\tau>0.
\end{equation}
Proceeding by contradiction, we assume that
$$\int_0^t\vert Q^{\frac{1}{2}}e^{\tau B^T}\xi\vert^{\alpha}d\tau = 0,$$
where $t>0$ and $\xi\in\mathbb{S}^{n-1}$. Since $\tau\mapsto\vert Q^{\frac{1}{2}}e^{\tau B^T}\xi\vert^{\alpha}$ is continuous on $[0,t]$, it follows that
\begin{gather}\label{28022018E1}
	\forall\tau\in[0,t],\quad Q^{\frac{1}{2}}e^{\tau B^T}\xi = 0.
\end{gather}
By differentiating the identity \eqref{28022018E1} with respect to the $\tau$-variable and evaluating at $\tau=0$, we deduce that
$$\forall k\in\{0,\ldots,r\},\quad Q^{\frac{1}{2}}(B^T)^k\xi = 0.$$
We obtain from \eqref{05052018E4} that $\xi = 0$. This proves \eqref{31072018E2} contradicting that $\xi\in\mathbb{S}^{n-1}$. By using the compactness of $\mathbb{S}^{n-1}$ and the continuity property with respect to the $\xi$-variable, it follows that the term $M^s_t$ is actually well-defined and satisfies $0<M^s_t<+\infty$. 
\end{proof}

We aim at studying the asymptotics of the term $M^s_t$ as $t$ tends to $0^+$. First, we set for all $\xi\in\mathbb{S}^{n-1}$,
$$M^s_{t,\xi} = \left[\int_0^t\vert Q^{\frac{1}{2}}e^{\tau B^T}\xi\vert^2d\tau\right]^{\frac{1}{2}}\left[\int_0^t\vert Q^{\frac{1}{2}}e^{\tau B^T}\xi\vert^{2s}d\tau\right]^{-\frac{1}{2s}}.$$
As an insight, we begin by studying the term $M^s_{t,\xi}$ for small times $t>0$. To that end, we consider
$$k_{\xi} = \min\left\{0\le k\le r,\quad Q^{\frac{1}{2}}(B^T)^k\xi\ne 0\right\},$$
which is well-defined by definition of $r$ in \eqref{05052018E4}. On the one hand, we observe that
$$\int_0^t\vert Q^{\frac{1}{2}}e^{\tau B^T}\xi\vert^2d\tau\underset{t\rightarrow0}{\sim}\vert Q^{\frac{1}{2}}(B^T)^{k_{\xi}}\xi\vert^2\int_0^t\frac{\tau^{2k_{\xi}}}{(k_{\xi}!)^2}d\tau
= \vert Q^{\frac{1}{2}}(B^T)^{k_{\xi}}\xi\vert^2\frac{t^{1+2k_{\xi}}}{(1+2k_{\xi})(k_{\xi}!)^2}.$$
Similarly, we have
$$\int_0^t\vert Q^{\frac{1}{2}}e^{\tau B^T}\xi\vert^{2s}d\tau\underset{t\rightarrow0}{\sim} \vert Q^{\frac{1}{2}}(B^T)^{k_{\xi}}\xi\vert^{2s}\frac{t^{1+2k_{\xi}s}}{(1+2k_{\xi}s)(k_{\xi}!)^{2s}},$$
and as a consequence, it follows that
$$M^s_{t,\xi}\underset{t\rightarrow0}{\sim}\frac{(1+2k_{\xi}s)^{\frac{1}{2s}}}{(1+2k_{\xi})^{\frac{1}{2}}}\ t^{\frac{1}{2}-\frac{1}{2s}}.$$
Unfortunately, some numerics suggest that the previous convergence does not stand uniformly on $\xi\in\mathbb{S}^{n-1}$, and therefore, the study of the term $M^s_t$ when $t\rightarrow0^+$ requires a more careful analysis. The Jensen inequality provides a first global estimate :

\begin{lem}\label{21052018L1} For all $t>0$, we have
$$M^s_t\le t^{\frac{1}{2}-\frac{1}{2s}}\quad \text{when $s\geq1$,}\quad\text{and}\quad M^s_t\geq t^{\frac{1}{2}-\frac{1}{2s}}\quad \text{when $0<s\le1$}.$$
\end{lem}

\begin{proof} When $s\geq1$, we deduce from Jensen's inequality that for all $t>0$ and $\xi\in\mathbb{S}^{n-1}$,
$$\left[\frac{1}{t}\int_0^t\vert Q^{\frac{1}{2}}e^{\tau B^T}\xi\vert^2d\tau\right]^{\frac{1}{2}}\le\left[\frac{1}{t}\int_0^t\vert Q^{\frac{1}{2}}e^{\tau B^T}\xi\vert^{2s}d\tau\right]^{\frac{1}{2s}},$$
and therefore $M^s_t\le t^{\frac{1}{2}-\frac{1}{2s}}$. Similarly, $M^s_t\geq t^{\frac{1}{2}-\frac{1}{2s}}$ when $0<s\le 1$. 
\end{proof}

To deal with the case when $0<s<1$, we shall use the following instrumental lemma :

\begin{lem} \label{lemma_algebra} Let $E$ be a real finite-dimensional vector space and $L_1,L_2:E\rightarrow\mathbb{R}_+$ be two continuous functions satisfying for all $j\in\{1,2\}$,
$$\forall\lambda\geq 0, \forall P\in E,\quad L_j(\lambda P) = \lambda L_j(P),$$
and
$$\forall P\in E\setminus\{0\},\quad L_j(P) >0.$$
Then, there exists a positive constant $c>0$ such that
$$\forall P\in E,\quad L_1(P) \le cL_2(P).$$
\end{lem}

\begin{proof} Let $\Vert\cdot\Vert$ be a norm on $E$ and $\mathbb{S}$ the associated unit sphere. Since $E$ is finite-dimensional, $\mathbb{S}$ is compact. Moreover, $L_1$ and $L_2$ are continuous and positive on $\mathbb{S}$, and as a consequence, by homogeneity,
$$\exists c_1,c_2>0, \forall P\in E,\quad L_1(P)\le c_1\Vert P\Vert\ \text{and}\ \Vert P\Vert  \le c_2L_2(P).$$
We deduce that
$$\forall P\in E,\quad L_1(P)\le c_1c_2 L_2(P).$$
\end{proof}

The next lemma is a direct application of Lemma \ref{lemma_algebra}. It only deals with the case when the matrix $B$ is nilpotent but its proof contains the main ideas that will be used to tackle the general case :

\begin{lem}\label{21052018L2} When $B$ is nilpotent, we have that for all $s>0$, there exists a positive constant $c>0$ such that for all $t>0$,
$$M^s_t\le ct^{\frac{1}{2}-\frac{1}{2s}}.$$
\end{lem}

\begin{proof} For all $t>0$ and $\xi\in\mathbb{S}^{n-1}$, we consider the term
$$M^s_{t,\xi} = \left[\int_0^t\vert Q^{\frac{1}{2}}e^{\tau B^T}\xi\vert^2d\tau\right]^{\frac{1}{2}}\left[\int_0^t\vert Q^{\frac{1}{2}}e^{\tau B^T}\xi\vert^{2s}d\tau\right]^{-\frac{1}{2s}},$$
and the function
$$f_{t,\xi}(\alpha) = Q^{\frac{1}{2}} e^{ t \alpha B^T} \xi,\quad \alpha\in[0,1].$$
Let $k$ be the index of $B$. Since $B^T$ is also nilpotent with index $k$, we have that
$$\forall t>0,\forall\xi\in\mathbb{S}^{n-1},\quad f_{t,\xi} \in (\mathbb{R}_k[X])^n.$$
It follows from Lemma \ref{lemma_algebra} applied with $E = (\mathbb{R}_k[X])^n$ and the functions
$$L_1(f) = \left[\int_0^1\vert f(\alpha)\vert^2d\alpha\right]^{\frac{1}{2}}\quad\text{and}\quad L_2(f) = \left[\int_0^1\vert f(\alpha)\vert^{2s}d\alpha\right]^{\frac{1}{2s}},$$ 
that there exists a positive constant $c>0$ such that for all $t>0$ and $\xi\in\mathbb{R}^{n-1}$,
$$t^{\frac{1}{2s}-\frac{1}{2}}M^s_{t,\xi} 
= \left[\int_0^1\vert f_{t,\xi}(\alpha)\vert^2d\alpha\right]^{\frac{1}{2}}\left[\int_0^1\vert f_{t,\xi}(\alpha)\vert^{2s}d\alpha\right]^{-\frac{1}{2s}}\le c.$$
This ends the proof of Lemma \ref{21052018L2}.
\end{proof}

The next lemma is an adaptation of the previous one that allows to drop the assumption on the nilpotency of the matrix $B$ but only in the asymptotics when $t$ tends to $0^+$.

\begin{prop}\label{jenvoiedupateenzero} For all $s>0$, there exist $c>0$ and $0<t_0<1$ such that for all $0<t< t_0$,
$$M^s_t\le ct^{\frac{1}{2}-\frac{1}{2s}}.$$
\end{prop}

\begin{proof} For all $t>0$ and $\xi\in\mathbb{S}^{n-1}$, we consider anew
$$M^s_{t,\xi} = \left[\int_0^t\vert Q^{\frac{1}{2}}e^{\tau B^T}\xi\vert^2d\tau\right]^{\frac{1}{2}}\left[\int_0^t\vert Q^{\frac{1}{2}}e^{\tau B^T}\xi\vert^{2s}d\tau\right]^{-\frac{1}{2s}}.$$
Let $P_{t,\xi}$ and $R_{t,\xi}$ be defined for all $\alpha\in[0,1]$ by
$$P_{t,\xi}(\alpha)= \sum_{k=0}^r \alpha^k \frac{t^k}{k!}Q^{\frac{1}{2}} (B^T)^k \xi\quad\text{and}\quad R_{t,\xi}(\alpha) = Q^{\frac{1}{2}} e^{t\alpha B^T} \xi - P_{t,\xi}(\alpha).$$
We use Lemma \ref{lemma_algebra} with $E = (\mathbb{R}_{r}[X])^n$ and the functions
$$L_1(P) = \left[\int_0^1\vert P(\alpha)\vert^2d\alpha\right]^{\frac{1}{2}}\quad\text{and}\quad L_2(P) = \left[\int_0^1\vert P(\alpha)\vert^{2s}d\alpha\right]^{\frac{1}{2s}},$$
to obtain that there exists a constant $c>0$ such that for all $t>0$ and $\xi\in\mathbb{S}^{n-1}$,
\begin{align}\label{estimationlapluslongdelunivers}
	t^{\frac{1}{2s}-\frac{1}{2}}M^s_{t,\xi} & = \left[ \int_0^1\vert Q^{\frac{1}{2}} e^{t\alpha B^T}\xi\vert^2 d\alpha \right]^{\frac{1}{2}}\left[\int_0^1\vert Q^{\frac{1}{2}} e^{t\alpha B^T}\xi\vert^{2s}d\alpha\right]^{-\frac{1}{2s}} \\[5pt]
	& \le c \left[\frac{\displaystyle\int_0^1\vert Q^{\frac{1}{2}}e^{t\alpha B^T}\xi\vert^2 d\alpha}{\displaystyle\int_0^1\vert P_{t,\xi}(\alpha)\vert^{2}d\alpha}\right]^{\frac{1}{2}} \left[\frac{\displaystyle\int_0^1\vert P_{t,\xi}(\alpha)\vert^{2s}d\alpha}{\displaystyle\int_0^1\vert Q^{\frac{1}{2}}e^{t\alpha B^T}\xi\vert^{2s}d\alpha}\right]^{\frac{1}{2s}}. \nonumber
\end{align}
We aim at establishing uniform upper bounds with respect to $\xi$ and $t$ for these two factors. To that end, we equip $(\mathbb{R}_{r}[X])^n$ of the Hardy's norm $\| \cdot \|_{\mathcal{H}^\infty}$ defined by
$$\forall P \in(\mathbb{R}_{r}[X])^n,\quad \Vert P\Vert_{\mathcal{H}^\infty} = \max_{k\in\{0,\ldots,r\}} \frac{\vert P^{(k)}(0)\vert}{k!}.$$
We deduce anew from Lemma \ref{lemma_algebra} applied with $E=(\mathbb{R}_{r}[X])^n$ and the functions $\Vert\cdot\Vert_{\mathcal{H}^{\infty}}$ and
$$L(P) = \left[\int_0^1 |P(\alpha)|^p d\alpha \right]^{\frac{1}{p}},\quad p\in\{2,2s\},$$
that
\begin{gather}\label{ineq_equiv}
	\forall p\in\{2,2s\}, \exists c_p>0, \forall P \in (\mathbb{R}_{r}[X])^n,\quad \Vert P\Vert_{\mathcal{H}^{\infty}} \le c_p \left[\int_0^1\vert P(\alpha)\vert^p d\alpha \right]^{\frac{1}{p}}.  
\end{gather}
According to \eqref{05052018E4}, we notice that
$$\forall\xi\in\mathbb{S}^{n-1},\exists k_{\xi}\in\{0,\ldots,r\},\quad \max_{k\in\{0,\ldots,r\}}\frac{\vert Q^{\frac{1}{2}}(B^T)^k\xi\vert}{k!}
\geq\frac{\vert Q^{\frac{1}{2}}(B^T)^{k_{\xi}}\xi\vert}{(k_{\xi})!}>0,$$
and since the function
$$\xi\in\mathbb{S}^{n-1}\mapsto\max_{k\in\{0,\ldots,r\}}\frac{\vert Q^{\frac{1}{2}}(B^T)^k\xi\vert}{k!} \quad
\text{is continuous on $\mathbb{S}^{n-1}$},$$
we deduce by compactness that there exists a positive constant $\varepsilon>0$ such that
$$\forall \xi \in \mathbb{S}^{n-1},\quad \Vert P_{1,\xi}\Vert_{\mathcal{H}^{\infty}}\geq\varepsilon.$$
It follows that
$$\forall t\in(0,1], \forall\xi\in\mathbb{S}^{n-1},\quad \Vert P_{t,\xi} \Vert_{\mathcal{H}^{\infty}} \geq \varepsilon t^r, $$
and we deduce from \eqref{ineq_equiv} that
\begin{gather}\label{ineq_Kalman}
	\forall p\in\{2,2s\}, \forall t\in(0,1], \forall\xi\in\mathbb{S}^{n-1},\quad \varepsilon t^r \le c_p\left[\int_0^1\vert P_{t,\xi}(\alpha)\vert^p d\alpha \right]^{\frac{1}{p}}.
\end{gather}
On the other hand, it follows from the integral version of Taylor's formula that
$$\forall t>0, \forall \xi\in\mathbb{S}^{n-1},\forall\alpha\in[0,1],\quad R_{t,\xi}(\alpha) = \frac{(t\alpha)^{r+1}}{r!}\int_0^1(1-\theta)^rQ^{\frac{1}{2}}(B^T)^{r+1}e^{t\alpha\theta B^T}\xi d\theta.$$
Therefore, there exists $M>0$ such that
\begin{gather}\label{ineq_remainder}
  \forall t\in(0,1], \forall \xi\in\mathbb{S}^{n-1},\quad \Vert R_{t,\xi}\Vert_{L^{\infty}[0,1]} \le M t^{r+1}.
\end{gather}
With these estimates, we can obtain upper bounds on the two factors of the right-hand-side of \eqref{estimationlapluslongdelunivers}. \\[5pt]
\textbf{1.} Applying the triangle inequality for the $L^2$ norm, we have 
$$\forall t>0,\forall\xi\in\mathbb{S}^{n-1},\quad\left[\frac{\displaystyle\int_0^1\vert Q^{\frac{1}{2}}e^{t\alpha B^T}\xi\vert^2 d\alpha }{ \displaystyle \int_0^1 \vert P_{t,\xi}(\alpha)\vert^{2} d\alpha}\right]^{\frac{1}{2}} 
\le 1 + \left[\frac{\displaystyle\int_0^1\vert R_{t,\xi}(\alpha)\vert^2 d\alpha}{\displaystyle\int_0^1\vert P_{t,\xi}(\alpha)\vert^{2} d\alpha}\right]^{\frac{1}{2}}.$$
According to \eqref{ineq_Kalman} and \eqref{ineq_remainder}, we get that for all $\xi\in\mathbb{S}^{n-1}$ and $0<t\le1$,
\begin{gather}\label{22052018E1}
	\left[\frac{\displaystyle\int_0^1\vert Q^{\frac{1}{2}}e^{t\alpha B^T}\xi\vert^2 d\alpha}{\displaystyle \int_0^1\vert P_{t,\xi}(\alpha)\vert^{2} d\alpha}\right]^{\frac{1}{2}} 
	\le 1+\frac{c_2Mt^{r+1}}{\varepsilon t^r} = 1 + \frac{c_2 M}{\varepsilon}t
	\le1 + \frac{c_2 M}{\varepsilon}.
\end{gather}
\textbf{2.} We apply Lemma \ref{2052018L1} with $q=2s$ to derive that
$$\frac{\displaystyle\int_0^1\vert Q^{\frac{1}{2}}e^{t\alpha B^T}\xi\vert^{2s}d\alpha}{\displaystyle\int_0^1\vert P_{t,\xi}(\alpha)\vert^{2s} d\alpha}
\geq2^{-(2s-1)_+} - \frac{\displaystyle\int_0^1\vert R_{t,\xi}(\alpha)\vert^{2s}d\alpha}{\displaystyle\int_0^1\vert P_{t,\xi}(\alpha)\vert^{2s}d\alpha}.$$
Yet, it follows from \eqref{ineq_Kalman} and \eqref{ineq_remainder} that for all $\xi\in\mathbb{S}^{n-1}$ and $0<t\le1$,
$$\frac{\displaystyle\int_0^1\vert R_{t,\xi}(\alpha)\vert^{2s}d\alpha}{\displaystyle\int_0^1\vert P_{t,\xi}(\alpha)\vert^{2s}d\alpha}
\le\left[\frac{c_{2s}Mt^{r+1}}{\varepsilon t^r}\right]^{2s} = \left[\frac{c_{2s}M}{\varepsilon}t\right]^{2s},$$
from which we deduce that
\begin{equation}\label{jeserviraiunjour}
	\frac{\displaystyle\int_0^1\vert Q^{\frac{1}{2}}e^{t\alpha B^T}\xi\vert^{2s}d\alpha}{\displaystyle\int_0^1\vert P_{t,\xi}(\alpha)\vert^{2s}d\alpha}
	\geq 2^{-(2s-1)_+}  - \left[\frac{c_{2s} M}{\varepsilon}t \right]^{2s}.
\end{equation}
It follows from \eqref{jeserviraiunjour} that there exist some positive constants $c_0>0$ and $0<t_0<1$ such that for all $\xi\in\mathbb{R}^n$ and $0<t<t_0$,
\begin{equation}\label{23092018E1}
	\left[\frac{\displaystyle\int_0^1\vert Q^{\frac{1}{2}}e^{t\alpha B^T}\xi\vert^{2s}d\alpha}{\displaystyle\int_0^1\vert P_{t,\xi}(\alpha)\vert^{2s}d\alpha}\right]^{\frac{1}{2s}}
	\geq c_0.
\end{equation}
As a consequence of \eqref{estimationlapluslongdelunivers}, \eqref{22052018E1} and \eqref{23092018E1}, there exists a positive constant $c_1>0$ such that
$$\forall t\in(0,t_0), \forall \xi\in\mathbb{S}^{n-1},\quad M^s_{t,\xi}\le c_1t^{\frac12 - \frac1{2s}}.$$
This ends the proof of Proposition \ref{jenvoiedupateenzero}
\end{proof}

\subsection{Proof of Theorem \ref{30082018T1}} The above asymptotics of the term $M_t^s$ allow to refine the results of Proposition \ref{15022018P2} and to prove Theorem \ref{30082018T1}. First, it follows from Proposition \ref{15022018P2} and Proposition \ref{jenvoiedupateenzero} that there exist some positive constants $C_1>1$ and $0<t_0<1$ such that for all $k\in\{0,\ldots,r\}$, $q>0$, $0<t< t_0$ and $u\in L^2(\mathbb{R}^n)$,
\begin{gather}\label{21052018E9}
	\Vert\vert \mathbb{P}_kD_x\vert^q e^{-t\mathcal{P}}u\Vert_{L^2(\mathbb{R}^n)}
	\le \frac{C_1^{1+q}}{t^{q(\frac{1}{2s}+k)}}\ e^{\frac{1}{2}\Tr(B)t}\ q^{\frac{q}{2s}}\ \Vert u\Vert_{L^2(\mathbb{R}^n)}.
\end{gather}
We can consider $C_2>1$ a positive constant satisfying that for all $k\in\{0,\ldots,r\}$, $q>0$ and $t\in(0,t_0)$,
\begin{gather}\label{21052018E10}
	1\le \frac{C_2^{1+q}}{t^{q(\frac{1}{2s}+k)}}\ q^{\frac{q}{2s}}\quad \text{and}\quad 
	\frac{C^{1+q}_1}{t^{q(\frac{1}{2s}+k)}}\ q^{\frac{q}{2s}}\le \frac{C^{1+q}_2}{t^{q(\frac{1}{2s}+k)}}\ q^{\frac{q}{2s}}.
\end{gather}
Then, it follows from \eqref{21052018E9}, \eqref{21052018E10}, Theorem \ref{06022018T1}, Lemma \ref{05022018L2} and the Plancherel theorem that for all $k\in\{0,\ldots,r\}$, $q>0$, $0<t<t_0$ and $u\in L^2(\mathbb{R}^n)$,
\begin{align*}
	\Vert\langle\mathbb{P}_kD_x\rangle^q e^{-t\mathcal{P}}u\Vert_{L^2(\mathbb{R}^n)}
	& \le 2^{(q-1)_+}\left[\Vert\vert\mathbb P_kD_x\vert^q e^{-t\mathcal{P}}u\Vert_{L^2(\mathbb{R}^n)} + \Vert e^{-t\mathcal{P}}u\Vert_{L^2(\mathbb{R}^n)}\right]\\[5pt]
	& \le 2^{(q-1)_+}\left[\frac{C^{1+q}_1}{t^{q(\frac{1}{2s}+k)}}\ q^{\frac{q}{2s}} + 1\right]e^{\frac{1}{2}\Tr(B)t}\ \Vert u\Vert_{L^2(\mathbb{R}^n)} \\[5pt]
	& \le 2^{1+(q-1)_+}\frac{C^{1+q}_2}{t^{q(\frac{1}{2s}+k)}}\ e^{\frac{1}{2}\Tr(B)t}\ q^{\frac{q}{2s}}\ \Vert u\Vert_{L^2(\mathbb{R}^n)}. 
\end{align*}
This ends the proof of Theorem \ref{30082018T1}.

\subsection{Proof of Corollary \ref{28092018C1}} To end this section, we prove Corollary \ref{28092018C1}. Notice that the second estimate in Corollary \ref{28092018C1} is a straightforward consequence of Theorem \ref{30082018T1}. In order to reformulate the results of Theorem \ref{30082018T1} while using the matrices $Q^{\frac{1}{2}}(B^T)^k$ instead of the orthogonal projections $\mathbb{P}_k$, we begin by proving that there exists a positive constant $C_1>1$ such that for all $k\in\{0,\ldots,r\}$, $q>0$ and $\xi\in\mathbb R^n$,
\begin{equation}\label{03092018E2}
	\langle Q^{\frac{1}{2}}(B^T)^k\xi\rangle^q\le C_1^{1+q}\langle\mathbb P_k\xi\rangle^q.
\end{equation}
It follows from \eqref{01062018E2} that for all $k\in\{0,\ldots,r\}$, the canonical Euclidean orthogonal complement of the vector space $V_k$ is given by
$$V_k^{\perp} = \Ker(Q^{\frac{1}{2}})\cap\Ker(Q^{\frac{1}{2}}B^T)\cap\ldots\cap\Ker(Q^{\frac{1}{2}}(B^T)^k).$$
As a consequence, the following estimates hold for all $k\in\{0,\ldots,r\}$, $q>0$ and $\xi\in\mathbb R^n$,
\begin{multline*}
	\sum_{j=0}^k\langle Q^{\frac{1}{2}}(B^T)^j\xi\rangle^q = \sum_{j=0}^k\langle Q^{\frac{1}{2}}(B^T)^j\mathbb P_k\xi\rangle^q
	\le \left[\sum_{j=0}^k\max(1,\Vert Q^{\frac{1}{2}}(B^T)^j\Vert^q)\right]\langle\mathbb P_k\xi\rangle^q \\
	\le (r+1)\max_{0\le j\le r}(1,\Vert Q^{\frac{1}{2}}(B^T)^j\Vert^q)\langle\mathbb P_k\xi\rangle^q.
\end{multline*}
This proves \eqref{03092018E2}, since we have that for all $k\in\{0,\ldots,r\}$, $q>0$ and $\xi\in\mathbb R^n$,
$$\langle Q^{\frac{1}{2}}(B^T)^k\xi\rangle^q\le\sum_{j=0}^k\langle Q^{\frac{1}{2}}(B^T)^j\xi\rangle^q.$$ 
As a consequence of \eqref{03092018E2}, Theorem \ref{30082018T1} and the Plancherel theorem, we then deduce that there exists some positive constants $C_1,C_2>1$ and $0<t_0<1$ such that for all $k\in\{0,\ldots,r-1\}$, $q>0$, $0<t<t_0$ and $u\in L^2(\mathbb{R}^n)$,
\begin{align*}
	\Vert\langle Q^{\frac{1}{2}}(B^T)^kD_x\rangle^qe^{-t\mathcal{P}}u\Vert_{L^2(\mathbb{R}^n)} 
	& \le C_1^{1+q}\sum_{j=0}^k\Vert\langle\mathbb{P}_j D_x\rangle^qe^{-t\mathcal{P}}u\Vert_{L^2(\mathbb{R}^n)} \\[5pt]
	& \le C_1^{1+q}\sum_{j=0}^k\frac{C_2^{1+q}}{t^{q(\frac{1}{2s}+j)}}\ e^{\frac{1}{2}\Tr(B)t}\ q^{\frac{q}{2s}}\ \Vert u\Vert_{L^2(\mathbb{R}^n)} \\[5pt]
	& \le (r+1)\frac{(C_1C_2)^{1+q}}{t^{q(\frac{1}{2s}+k)}}\ e^{\frac{1}{2}\Tr(B)t}\ q^{\frac{q}{2s}}\ \Vert u\Vert_{L^2(\mathbb{R}^n)}.
\end{align*}
This ends the proof of Corollary \ref{28092018C1}.

\section{Observability estimates for fractional Ornstein-Uhlenbeck semigroups}
\label{sec_OeffOUs}

In this section, we prove Theorem \ref{23052018T1} and Theorem \ref{20082019T1}.

\subsection{Proof of Theorem \ref{23052018T1}} This first subsection is devoted to the proof of Theorem \ref{23052018T1}. We consider $\mathcal{P}$ the fractional Ornstein-Uhlenbeck operator defined in \eqref{06022018E3} and equipped with the domain \eqref{19062018E1}. We assume that the Kalman rank condition \eqref{10052018E4} holds. Moreover, we consider the operator
$$\mathcal{P}_{co} = \mathcal{P} + \frac{1}{2}\Tr(B),$$
equipped with the domain $D(\mathcal{P})$. Let $\omega$ be a measurable subset of $\mathbb{R}^n$. To establish the observability estimate (\ref{06022018E2}), we use the following theorem established by K. Beauchard and K. Pravda-Starov in \cite{MR3732691} (Theorem 2.1), which is essentially a reformulation of a previous result due to L. Miller \cite{MR2679651} (involving a telescopic series), following the seminal ideas in \cite{MR1312710}.

\begin{thm}\label{08122017T3} Let $\omega$ be a measurable subset of $\mathbb{R}^n$ with positive Lebesgue measure, $(\pi_k)_{k\geq1}$ be a family of orthogonal projections defined on $L^2(\mathbb{R}^n)$ and $(e^{tA})_{t\geq0}$ be a contraction semigroup on $L^2(\mathbb{R}^n)$. Assume that there exist $c_1,c_2,a,b,t_0,m>0$ some positive constants with $a<b$ such that the following spectral inequality 
\begin{gather}\label{11050218E3}
	\forall u\in L^2(\mathbb{R}^n), \forall k\geq1,\quad \Vert\pi_k u\Vert_{L^2(\mathbb{R}^n)}\le e^{c_1k^a}\Vert\pi_ku\Vert_{L^2(\omega)},
\end{gather}
and the following dissipation estimate
\begin{gather}\label{11050218E4}
	\forall u\in L^2(\mathbb{R}^n), \forall k\geq1, \forall 0<t<t_0,\quad \Vert(1-\pi_k)(e^{tA}u)\Vert_{L^2(\mathbb{R}^n)}\le \frac{1}{c_2}e^{-c_2t^mk^b}\Vert u\Vert_{L^2(\mathbb{R}^n)},
\end{gather}
hold. Then, there exists a positive constant $C>1$ such that the following observability estimate holds
$$\forall T>0, \forall u\in L^2(\mathbb{R}^n),\quad \Vert e^{TA}u\Vert^2_{L^2(\mathbb{R}^n)}\le C\exp\left(\frac{C}{T^{\frac{am}{b-a}}}\right)\int_0^T\Vert e^{tA}u\Vert^2_{L^2(\omega)}\ dt.$$
\end{thm}

Notice that in \cite{MR3732691} (Theorem 2.1) the subset $\omega$ is assumed to be open, but the proof works the same when $\omega$ is only a measurable subset.

Let $\pi_k:L^2(\mathbb{R}^n)\rightarrow E_k$, $k\geq1$, be the orthogonal frequency cutoff projection onto the closed subspace
\begin{gather}\label{11052018E5}
	E_k = \{u\in L^2(\mathbb{R}^n),\quad \Supp\widehat{u}\subset[-k,k]^n\}.
\end{gather}
According to Theorem \ref{08122017T3}, it is sufficient to derive a spectral inequality as \eqref{11050218E3} and a dissipation estimate as \eqref{11050218E4} for the orthogonal projections $\pi_k$ to establish the observability estimate \eqref{06022018E2}.

\subsubsection{Spectral inequality} The following theorem is proved by O. Kovrijkine in \cite{MR1840110} (Theorem 3).

\begin{thm}\label{11052018T1} There exists a universal constant $K$ depending only on the dimension $n$ that may be assumed to be greater or equal to $e$ such that for any $J$ a parallelepiped with sides parallel to the coordinate axis and of positive lengths $b_1,\ldots,b_n$ and $\omega$ a $(\gamma,a)$-thick set, then
$$\forall u\in L^2(\mathbb{R}^n),\ \Supp\widehat{u}\subset J,\quad\Vert u\Vert_{L^2(\mathbb{R}^n)}\le \left(\frac{K^n}{\gamma}\right)^{K(\langle a,b\rangle + n)}\Vert u\Vert_{L^2(\omega)},$$
where $b=(b_1,\ldots,b_n)$.
\end{thm}

We now assume that the set $\omega$ is thick. Let $u\in L^2(\mathbb{R}^n)$ and $k\geq1$. It follows from the definition of $\pi_k$, see \eqref{11052018E5}, that $\pi_ku\in L^2(\mathbb{R}^n)$ and $\widehat{\pi_ku}$ is supported in $[-k,k]^n$. Therefore, we deduce from Theorem \ref{11052018T1} that
\begin{align}\label{11052018E6}
	\forall u\in L^2(\mathbb{R}^n), \forall k\geq1,\quad \Vert\pi_ku\Vert_{L^2(\mathbb{R}^n)}\le e^{c_1k}\Vert\pi_ku\Vert_{L^2(\omega)},
\end{align}
where
$$c_1 = \left(\ln\left[\left(\frac{K^n}{\gamma}\right)^{nK}\right]\right)_+ + \left(\ln\left[\left(\frac{K^n}{\gamma}\right)^{2K(a_1+\ldots+a_n)}\right]\right)_++1>0,$$
and where $x_+ = \max(x,0)$ for all $x\in\mathbb{R}$.

\subsubsection{Dissipation estimate} As a consequence of Theorem \ref{30082018T1}, there exist some positive constants $c>1$ and $0<t_0<1$ such that for all $N\geq0$, $0<t<t_0$ and $u\in L^2(\mathbb{R}^n)$,
\begin{equation}\label{08062018E2}
	\Vert\vert D_x\vert^{2sN}e^{-t\mathcal{P}_{co}}u\Vert_{L^2(\mathbb{R}^n)}\le \frac{c^{1+2sN}}{t^{2sN\Gamma}}\ (2sN)^N\ \Vert u\Vert_{L^2(\mathbb{R}^n)},
\end{equation}
with the convention that $0^0 = 1$ and where we set $\Gamma = \frac{1}{2s} + r$.
It follows from \eqref{08062018E2} that for all $0<t<t_0$ and $u\in L^2(\mathbb{R}^n)$,
\begin{align}\label{20082019E2}
	\Vert\exp\left[\frac{1}{4es}\left[\frac{t^{\Gamma}}{c}\right]^{2s}\vert D_x\vert^{2s}\right]e^{-t\mathcal{P}_{co}}u\Vert_{L^2(\mathbb{R}^n)} & \le \sum_{N=0}^{+\infty}\frac{1}{2^N}\left[\frac{t^{\Gamma}}{c}\right]^{2sN}\frac{1}{(2es)^NN!}\Vert\vert D_x\vert^{2sN}e^{-t\mathcal{P}_{co}}u\Vert_{L^2(\mathbb{R}^n)}\\[5pt]
	%\le &\ \sum_{N=0}^{+\infty}\frac{1}{2^N}\left[\frac{t^{\Gamma}}{c}\right]^{2sN}\frac{1}{(2es)^NN!}\Vert\vert D_x\vert^{2sN}e^{-t\mathcal{P}_{co}}u\Vert_{L^2(\mathbb{R}^n)}. \nonumber \\[5pt]
	& \le c\sum_{N=0}^{+\infty}\frac{1}{2^N}\frac{(2sN)^N}{(2es)^NN!}\Vert u\Vert_{L^2(\mathbb{R}^n)}. \nonumber
\end{align}
Moreover, we have $(2sN)^N\le (2es)^N\ N!$ for all $N\geq0$, see e.g. formula (0.3.12) in \cite{MR2668420}, and \eqref{20082019E2} implies the following estimates
\begin{gather}\label{05022015E13}
	\forall t\in(0,t_0),\forall u\in L^2(\mathbb{R}^n),\quad \Vert e^{Ct^{2s\Gamma}\vert D_x\vert^{2s}}e^{-t\mathcal{P}_{co}}u\Vert_{L^2(\mathbb{R}^n)}\le 2c\Vert u\Vert_{L^2(\mathbb{R}^n)},
\end{gather}
where we set $C = \frac{1}{4ec^{2s}s}$.
It follows from \eqref{05022015E13} and the Plancherel theorem that for all $k\geq1$,
\begin{align*}
	\Vert(1-\pi_k)e^{-t\mathcal{P}_{co}}u\Vert_{L^2(\mathbb{R}^n)} 
	& = \frac{1}{(2\pi)^{\frac{n}{2}}}\Vert\mathbbm{1}_{\mathbb{R}^n\setminus[-k,k]^n}\ \widehat{e^{-t\mathcal{P}_{co}}u}\Vert_{L^2(\mathbb{R}^n)} \\[5pt]
	& = \frac{1}{(2\pi)^{\frac{n}{2}}}\Vert\mathbbm{1}_{\mathbb{R}^n\setminus[-k,k]^n}\ e^{-Ct^{2s\Gamma}\vert\xi\vert^{2s}} e^{Ct^{2s\Gamma}\vert\xi\vert^{2s}}\widehat{e^{-t\mathcal{P}_{co}}u}\Vert_{L^2(\mathbb{R}^n)} \\[5pt]
	& \le e^{-Ct^{2s\Gamma}k^{2s}}\Vert e^{Ct^{2s\Gamma}\vert D_x\vert^{2s}}e^{-t\mathcal{P}_{co}}u\Vert_{L^2(\mathbb{R}^n)} \\[5pt]
	& \le 2c e^{-Ct^{2s\Gamma}k^{2s}}\Vert u\Vert_{L^2(\mathbb{R}^n)}.
\end{align*}
Setting
$$c_2 = \min\left(\frac{1}{2c},C\right)\quad \text{and\quad $m = 2s\Gamma$},$$
we proved that for all $k\geq1$, $0<t<t_0$ and $u\in L^2(\mathbb{R}^n)$,
\begin{align}\label{08122017E20}
	\Vert(1-\pi_k)e^{-t\mathcal{P}_{co}}u\Vert_{L^2(\mathbb{R}^n)}\le\frac{1}{c_2}\ e^{-c_2t^mk^{2s}}\Vert u\Vert_{L^2(\mathbb{R}^n)}.
\end{align}

\subsubsection{Observability estimate} Since $2s>1$, we deduce from \eqref{11052018E6}, (\ref{08122017E20}) and Theorem \ref{08122017T3} that there exists a positive constant $C>1$ such that 
$$\forall T>0, \forall u\in L^2(\mathbb{R}^n),\quad \Vert e^{-T\mathcal{P}_{co}}u\Vert^2_{L^2(\mathbb{R}^n)}\le C\exp\left(\frac{C}{T^{\frac{1+2rs}{2s-1}}}\right)\int_0^T\Vert e^{-t\mathcal{P}_{co}}u\Vert^2_{L^2(\omega)}\ dt.$$
This proves the observability estimate \eqref{06022018E2} and ends the proof of Theorem \ref{23052018T1}.

\subsection{Proof of Theorem \ref{20082019T1}} To end this section, we prove Theorem \ref{20082019T1}. The following proof is inspired by \cite{MR3816981} (Section 4). Let $T>0$ and $\omega\subset\mathbb{R}^n$ be a measurable subset. We assume that $\omega$ is not thick. Since the operator $(-\Delta_x)^s$ equipped with the domain $H^s(\mathbb{R}^n)$ is selfadjoint from Corollary \ref{16032018C2}, it follows from the Hilbert Uniqueness Method, see \cite{MR2302744} (Theorem 2.44), that the fractional heat equation \eqref{24052018E4} is null-controllable from the set $\omega$ in time $T$ if and only if there exists a positive constant $C_T>0$ such that for all $g\in L^2(\mathbb R^n)$,  
\begin{equation}\label{29082018E2}
	\Vert e^{-T(-\Delta_x)^s}g\Vert^2_{L^2(\mathbb{R}^n)}\le C_T\int_0^T\Vert e^{-t(-\Delta_x)^s}g\Vert^2_{L^2(\omega)}\ dt.
\end{equation}
To prove Theorem \ref{20082019T1}, it is then sufficient to construct a sequence of functions $(g_{0,k})_k$ in $L^2(\mathbb{R}^n)$ such that the observability estimate \eqref{29082018E2} does not hold. Since the set $\omega$ is not thick, we have
$$\forall\gamma>0, \forall a\in(\mathbb R^*_+)^n, \exists \xi\in\mathbb{R}^n,\quad \vert\omega\cap(\xi+[0,a_1]\times\ldots\times[0,a_n])\vert<\gamma\prod_{j=1}^na_j.$$ 
Therefore, for all $k\geq1$, there exists $\xi_k\in\mathbb R^n$ such that the Lebesgue measure of the set $\omega\cap(\xi_k+[0,2k]^n)$ satisfies $\vert\omega\cap(\xi_k+[0,2k]^n)\vert < 1/k$. Setting $x_k = \xi_k+(k,\ldots,k)\in\mathbb{R}^n$, this inequality writes for all $k\geq1$ as
\begin{equation}\label{23082018E4}
	\vert\omega\cap\mathcal{B}(x_k,k)\vert<\frac1k,
\end{equation}
where $\mathcal{B}(x_k,k)\subset\mathbb R^n$ denotes the Euclidean ball centred in $x_k$ with radius $k$. With the points $x_k$, we construct the functions $g_{0,k} = f(\cdot-x_k)\in H^s(\mathbb{R}^n)$ for all $k\geq1$, where $f = \mathscr{F}^{-1}(e^{-\vert\xi\vert^{2s}})\in H^s(\mathbb R^n)$. Moreover, for all $k\geq1$, we define $g_k = e^{-t(-\Delta_x)^s}g_{0,k}\in L^2(\mathbb{R}^n)$. It follows from the definition of the functions $g_{0,k}$ that for all $k\geq1$, $t\geq0$ and $\xi\in\mathbb{R}^n$, $\widehat{g_k}(t,\xi) = e^{-i\langle x_k,\xi\rangle}e^{-(1+t)\vert\xi\vert^{2s}},$
and as a consequence, the functions $g_k$ are given by
\begin{equation}\label{23082018E5}
	\forall k\geq1, \forall t>0, \forall x\in\mathbb{R}^n,\quad g_k(t,x) = \frac{1}{(1+t)^{\frac{n}{2s}}}\ f\left(\frac{x-x_k}{(1+t)^{\frac{1}{2s}}}\right).
\end{equation}
It follows from \eqref{23082018E5} and the substitution rule that for all $k\geq1$,
\begin{multline}\label{29082018E5}
	\Vert g_k(T,\cdot)\Vert^2_{L^2(\mathbb{R}^n)} 
	= \frac{1}{(1+T)^{\frac ns}}\int_{\mathbb{R}^n}\left\vert f\left(\frac{x-x_k}{(1+T)^{\frac{1}{2s}}}\right)\right\vert^2\ dx \\
	= \frac{1}{(1+T)^{\frac ns}}\int_{\mathbb{R}^n}\left\vert f\left(\frac{x}{(1+T)^{\frac{1}{2s}}}\right)\right\vert^2\ dx>0.
\end{multline}
Therefore, the quantities $\Vert g_k(T,\cdot)\Vert_{L^2(\mathbb{R}^n)}$ are in fact independent of the parameter $k\geq1$. On the other hand, we deduce anew from the substitution rule that
\begin{multline*}
	\int_0^T\Vert g_k(t,\cdot)\Vert^2_{L^2(\omega)}\ dt 
	= \int_0^T\int_{\omega}\frac{1}{(1+t)^{\frac ns}}\left\vert f\left(\frac{x-x_k}{(1+t)^{\frac{1}{2s}}}\right)\right\vert^2\ dxdt \\[5pt]
	= \int_0^T\int_{\omega-x_k}\frac{1}{(1+t)^{\frac ns}}\left\vert f\left(\frac{x}{(1+t)^{\frac{1}{2s}}}\right)\right\vert^2\ dxdt.
\end{multline*}
By splitting the previous integral in two parts, we derive the following estimate :
\begin{multline}\label{23082018E6}
	\int_0^T\Vert g_k(t,\cdot)\Vert^2_{L^2(\omega)}\ dt 
	\le \int_0^T\int_{(\omega-x_k)\cap\mathcal B(0,k)}\frac{1}{(1+t)^{\frac ns}}\left\vert f\left(\frac{x}{(1+t)^{\frac{1}{2s}}}\right)\right\vert^2\ dxdt \\[5pt]
	+ \int_0^T\int_{\vert x\vert>k}\frac{1}{(1+t)^{\frac ns}}\left\vert f\left(\frac{x}{(1+t)^{\frac{1}{2s}}}\right)\right\vert^2\ dxdt.
\end{multline}
Now, we study one by one the two integrals appearing in the right-hand-side of \eqref{23082018E6} : \\[5pt]
\textbf{1.} First, it follows from the invariance by translation of the Lebesgue measure that
\begin{multline*}\label{23082018E7}
	\int_0^T\int_{(\omega-x_k)\cap\mathcal B(0,k)}\frac{1}{(1+t)^{\frac ns}}\left\vert f\left(\frac{x}{(1+t)^{\frac{1}{2s}}}\right)\right\vert^2\ dxdt  \\[5pt]
	\le T\ \Vert f\Vert^2_{L^{\infty}(\mathbb{R}^n)}\ \vert(\omega-x_k)\cap\mathcal B(0,k)\vert 
	= T\ \Vert f\Vert^2_{L^{\infty}(\mathbb{R}^n)}\ \vert \omega\cap\mathcal B(x_k,k)\vert,
\end{multline*}
and \eqref{23082018E4} implies the following convergence :
\begin{equation}\label{23082018E7}
	\int_0^T\int_{(\omega-x_k)\cap\mathcal B(0,k)}\frac{1}{(1+t)^{\frac ns}}\left\vert f\left(\frac{x}{(1+t)^{\frac{1}{2s}}}\right)\right\vert^2\ dxdt
	\le \frac{T}{k}\ \Vert f\Vert^2_{L^{\infty}(\mathbb{R}^n)}
	\underset{k\rightarrow+\infty}{\rightarrow}0.
\end{equation}
\textbf{2.} To control the second integral, we begin by checking that
\begin{equation}\label{29082018E4}
	\frac{1}{(1+t)^{\frac n{2s}}}\ f\left(\frac{x}{(1+t)^{\frac{1}{2s}}}\right)\in L^2([0,T]\times\mathbb R^n).
\end{equation}
It follows from the definition of the function $f$ and the substitution rule that for all $t\in[0,T]$ and $x\in\mathbb{R}^n$,
\begin{multline*}
	\frac{1}{(1+t)^{\frac n{2s}}}\ f\left(\frac{x}{(1+t)^{\frac{1}{2s}}}\right) 
	= \frac{1}{(2\pi)^n}\frac{1}{(1+t)^{\frac n{2s}}}\ \int_{\mathbb{R}^n}e^{i\langle x,\xi\rangle/(1+t)^{\frac1{2s}}}e^{-\vert\xi\vert^{2s}}\ d\xi \\[5pt]
	= \frac{1}{(2\pi)^n}\int_{\mathbb R^n} e^{i\langle x,\xi\rangle} e^{-(1+t)\vert\xi\vert^{2s}}\ d\xi
	= \mathscr{F}^{-1}_x(e^{-(1+t)\vert\xi\vert^{2s}})(t,x),
\end{multline*}
where $\mathscr{F}^{-1}_x$ denotes the inverse partial Fourier transform in the $x$ variable. Since the function $(t,\xi)\mapsto e^{-(1+t)\vert\xi\vert^{2s}}$ belongs to the space $L^2([0,T]\times\mathbb R^n)$, \eqref{29082018E4} is implied by the Plancherel theorem. Then, we deduce from the dominated convergence theorem that
\begin{equation}\label{23082018E8}
	\int_0^T\int_{\vert x\vert>k}\frac{1}{(1+t)^{\frac ns}}\left\vert f\left(\frac{x}{(1+t)^{\frac{1}{2s}}}\right)\right\vert^2\ dxdt \underset{k\rightarrow+\infty}{\rightarrow}0.
\end{equation}
As a consequence of \eqref{23082018E6}, \eqref{23082018E7} and \eqref{23082018E8}, the following convergence holds
\begin{equation}\label{29082018E6}
	\int_0^T\Vert g_k(t,\cdot)\Vert^2_{L^2(\omega)}\ dt\underset{k\rightarrow+\infty}{\rightarrow}0.
\end{equation}
We deduce from \eqref{29082018E5} and \eqref{29082018E6} that the observability estimate \eqref{29082018E2} does not hold. This ends the proof of Theorem \ref{20082019T1}.

\section{Global subelliptic estimates for fractional Ornstein-Uhlenbeck operators}
\label{sec_GLseffOUo}

In this section, we investigate the $L^2$ subelliptic properties enjoyed by fractional Ornstein-Uhlenbeck operators. Let $\mathcal{P}$ be the fractional Ornstein-Uhlenbeck operator defined in \eqref{06022018E3} and equipped with the domain \eqref{19062018E1}.  We assume that the Kalman rank condition \eqref{10052018E4} holds and we denote by $0\le r\le n-1$ the smallest integer satisfying \eqref{05052018E4}. Moreover, for all $0\le k\le r$, we consider $\mathbb P_k$ the orthogonal projection onto the vector subspace $V_k$ defined in \eqref{01062018E2}.

\subsection{Proof of Theorem \ref{28092018T1}} By using some results of interpolation theory as in \cite{MR3710672} (Subsection 2.4), we establish Theorem \ref{28092018T1}. First of all, we will exploit the following result, stated as Proposition 2.7 in \cite{MR3710672} and whose proof is given in \cite{MR2523200} (Corollary 5.13), which allows to localize the domain of an operator from the smoothing properties of the associated semigroup.
\begin{prop}[Proposition 2.7 in \cite{MR3710672}]\label{14062020P1} Let $X$ be a Hilbert space and $A:D(A)\subset X\rightarrow X$ be a maximal accretive operator such that $(-A,D(A))$ is the generator of a strongly continuous semigroup $(T(t))_{t\geq0}$. Assume that there exists a Banach space $E\subset X$, $\rho>1$ and $C>0$ such that
$$\forall t>0,\quad \Vert T(t)\Vert_{\mathcal L(X,E)}\le\frac C{t^{\rho}},$$
and that $t\mapsto T(t)u$ is measurable with values in $E$ for each $u\in X$. Then, the following continuous inclusion holds
$$D(A)\subset(X,E)_{\frac1\rho,2},$$
where $(X,E)_{\frac1\rho,2}$ denotes the space obtained by real interpolation.
\end{prop}

Let $k\in\{0,\ldots,r-1\}$. We consider the Fourier multiplier $\Lambda_k = \langle \mathbb P_kD_x\rangle$ and $\mathscr{H}_k$ the Hilbert space defined by
$$\mathscr{H}_k = \{u\in L^2(\mathbb{R}^n),\quad \Lambda_k^{\lfloor 2s\rfloor + 1} u\in L^2(\mathbb{R}^n)\},$$
equipped with the scalar product
$$\langle u,v\rangle_{\mathscr{H}_k} = \langle\Lambda_k^{\lfloor 2s\rfloor + 1} u,\Lambda_k^{\lfloor 2s\rfloor + 1} v\rangle_{L^2(\mathbb{R}^n)}.$$
It follows from Theorem \ref{30082018T1} that there exist some positive constants $C_1>1$ and $0<t_0<1$ such that for all $0<t<t_0$ and $u\in L^2(\mathbb{R}^n)$,
\begin{gather}\label{09022018E3}
	\Vert\Lambda_k^{\lfloor 2s\rfloor + 1} e^{-t\mathcal{P}}u\Vert_{L^2(\mathbb{R}^n)}\le \frac{C_1}{t^{1/\theta}}\ e^{\frac{1}{2}\Tr(B)t}\ \Vert u\Vert_{L^2(\mathbb{R}^n)},
\end{gather}
where
\begin{equation}\label{31072018E3}
	\theta = \left[\big(\lfloor2s\rfloor + 1\big)\left(\frac{1}{2s} + k\right)\right]^{-1}\in(0,1).
\end{equation}
Let $0<t_1<t_0$. It follows from Theorem \ref{06022018T1}, \eqref{09022018E3} and the semigroup property of the family of operators $(e^{-t\mathcal{P}})_{t\geq0}$ that for all $t\geq t_0$ and $u\in L^2(\mathbb{R}^n)$,
\begin{align}\label{23082018E2}
	\Vert\Lambda_k^{\lfloor 2s\rfloor + 1} e^{-t\mathcal{P}}u\Vert_{L^2(\mathbb{R}^n)}
	& = \Vert\Lambda_k^{\lfloor 2s\rfloor + 1} e^{-t_1\mathcal{P}}e^{-(t-t_1)\mathcal P}u\Vert_{L^2(\mathbb{R}^n)} \\[5pt]
	& \le \frac{C_1}{t_1^{1/\theta}}\ e^{\frac{1}{2}\Tr(B)t_1}\ \Vert e^{-(t-t_1)\mathcal P}u\Vert_{L^2(\mathbb{R}^n)} \nonumber \\[5pt]
	& \le \frac{C_1}{t_1^{1/\theta}}\ e^{\frac{1}{2}\Tr(B)t}\ \Vert u\Vert_{L^2(\mathbb{R}^n)}. \nonumber
\end{align}
We deduce from \eqref{09022018E3} and \eqref{23082018E2} that there exist some positive constants $C_2>0$ and $\mu>0$ such that for all $t>0$ and $u\in L^2(\mathbb{R}^n)$,
\begin{equation}\label{23082018E3}
	\Vert\Lambda_k^{\lfloor 2s\rfloor + 1} e^{-t\mathcal{P}}u\Vert_{L^2(\mathbb{R}^n)}\le \frac{C_2e^{\mu t}}{t^{1/\theta}}\ e^{\frac{1}{2}\Tr(B)t}\ \Vert u\Vert_{L^2(\mathbb{R}^n)}.
\end{equation}
Considering the operator 
\begin{equation}\label{30082018E5}
	\tilde{\mathcal P} = \mathcal{P} + \frac{1}{2}\Tr(B) + \mu,
\end{equation}
the inequality \eqref{23082018E3} can be written as
\begin{equation}\label{27082018E1}
	\forall t>0, \forall u\in L^2(\mathbb{R}^n),\quad \Vert e^{-t\tilde{\mathcal P}}u\Vert_{\mathscr{H}_k}\le\frac{C_2}{t^{1/\theta}}\ \Vert u\Vert_{L^2(\mathbb{R}^n)}.
\end{equation}
It follows from \eqref{27082018E1} and the strong continuity of the semigroup $(e^{-t\tilde{\mathcal P}})_{t\geq0}$ given by \eqref{30082018E5} and Theorem \ref{06022018T1} that for all $u\in L^2(\mathbb R^n)$, $t_0>0$ and $t>0$, we have 
$$\Vert e^{-(t+t_0)\tilde{\mathcal P}}u - e^{-t_0\tilde{\mathcal P}}u\Vert_{\mathcal H_k}
= \Vert e^{-t_0\tilde{\mathcal P}}\big(e^{-t\tilde{\mathcal P}}u - u\big)\Vert_{\mathcal H_k} \\
\le \frac{C}{t_0^{1/\theta}}\Vert e^{-t\tilde{\mathcal P}}u - u\Vert_{L^2(\mathbb R^n)}
\underset{t\rightarrow0}{\rightarrow}0.$$
This proves that for all $u\in L^2(\mathbb R^n)$, the function $t\in(0,+\infty)\mapsto e^{-t\tilde{\mathcal P}}u\in\mathcal H_k$ is continuous, and therefore measurable.
Moreover, we deduce from \eqref{19062018E1} and Corollary \ref{16032018C1} that the operator $\tilde{\mathcal P}$ equipped with the domain $D(\mathcal{P})$ is maximal accretive. According to \eqref{31072018E3}, Proposition \ref{14062020P1} shows that the following continuous inclusion holds
\begin{align}\label{26072017E2}
	D(\mathcal{P})\subset (L^2(\mathbb{R}^n),\mathscr{H}_k)_{\theta,2},
\end{align}
where $(L^2(\mathbb{R}^n),\mathscr{H}_k)_{\theta,2}$ denotes the space obtained by real interpolation. We notice that the Hilbert space $\mathscr{H}_k$ is dense in $L^2(\mathbb{R}^n)$. It follows from Corollary 4.37 in \cite{MR2523200} the following correspondence between real and complex interpolation spaces
\begin{align}\label{26072017E3}
	(L^2(\mathbb{R}^n),\mathscr{H}_k)_{\theta,2} = [L^2(\mathbb{R}^n),\mathscr{H}_k]_{\theta},
\end{align}
where $[L^2(\mathbb{R}^n),\mathscr{H}_k]_{\theta}$ stands for the space obtained by complex interpolation. In order to determine this space, we need the following result:

\begin{thm}[Theorem 4.36 in \cite{MR2523200}]\label{15062020T1} Let $(X,\langle\cdot,\cdot\rangle)$ be a Hilbert space and $A:D(A)\subset X\rightarrow X$ be a selfadjoint operator such that
$$\exists\delta>0, \forall x\in D(A),\quad\langle Ax,x\rangle\geq\delta\Vert x\Vert^2,$$
with $\Vert\cdot\Vert$ the norm associated with $\langle\cdot,\cdot\rangle$. We also consider $\alpha,\beta\in\mathbb C$ with $\Reelle\alpha\geq0$, $\Reelle\beta\geq0$. Then for every $\theta\in(0,1)$,
$$[D(A^{\alpha}),D(A^{\beta})]_{\theta} = D(A^{(1-\theta)\alpha+\theta\beta}).$$
\end{thm}

With $\mathscr{H}_k$ as the domain of the operator $\Lambda_k^{\lfloor 2s\rfloor + 1}$, we have that $\Lambda_k^{\lfloor 2s\rfloor + 1}$ is a positive selfadjoint operator satisfying
$$\forall u\in\mathscr{H}_k,\quad \langle\Lambda_k^{\lfloor 2s\rfloor + 1}u,u\rangle_{L^2(\mathbb{R}^n)}\geq\Vert u\Vert^2_{L^2(\mathbb{R}^n)}.$$ 
Thus, we deduce from Theorem \ref{15062020T1} that
\begin{equation}\label{26072017E4}
	[L^2(\mathbb{R}^n),\mathscr{H}_k]_{\theta} 
	= [D((\Lambda_k^{\lfloor2s\rfloor+1})^0),D(\Lambda_k^{\lfloor2s\rfloor+1})^1)]_{\theta}
	= D(\Lambda_k^{[\lfloor2s\rfloor+1]\theta}) 
	= D(\Lambda_k^{\frac{2s}{1+2sk}}).
\end{equation}
We therefore obtain from (\ref{26072017E2}), (\ref{26072017E3}) and (\ref{26072017E4}) that the following continuous inclusion holds
$$D(\mathcal{P})\subset D(\Lambda_k^{\frac{2s}{1+2ks}}).$$
This implies that there exists a positive constant $c_k>0$ such that
$$\forall u\in D(\mathcal{P}),\quad \Vert\Lambda_k^{\frac{2s}{1+2ks}}u\Vert_{L^2(\mathbb{R}^n)}\le c_k\left[\Vert\tilde{\mathcal P}u\Vert_{L^2(\mathbb{R}^n)} + \Vert u\Vert_{L^2(\mathbb{R}^n)}\right],$$
and we deduce from the definition of $\Lambda_k$ and \eqref{30082018E5} that
$$\forall u\in D(\mathcal{P}),\quad \Vert\langle\mathbb P_kD_x\rangle^{\frac{2s}{1+2ks}}u\Vert_{L^2(\mathbb{R}^n)}\le c_k\left[\Vert\mathcal Pu\Vert_{L^2(\mathbb{R}^n)} + \Vert u\Vert_{L^2(\mathbb{R}^n)}\right].$$
This ends the proof of Theorem \ref{28092018T1}.

\subsection{Proof of Corollary \ref{03102018C1}} By using Corollary \ref{28092018C1}, the proof of Corollary \ref{03102018C1} follows the very same arguments of interpolation theory as the ones used in the proof of Theorem \ref{28092018T1} for the spaces $L^2(\mathbb R^n)$ and
$$\mathscr{H}_k = \{u\in L^2(\mathbb{R}^n),\quad \langle Q^{\frac 12}(B^T)^kD_x\rangle^{\lfloor 2s\rfloor + 1} u\in L^2(\mathbb{R}^n)\},\quad 0\le k\le r-1.$$

\subsection{Proof of Corollary \ref{24052018C1}} Corollary \ref{24052018C1} is an immediate consequence of Theorem \ref{28092018T1}. By definition, $\mathbb P_0$ is the orthogonal projection onto the vector subspace $V_0 = \Ran Q^{\frac12}$. Moreover, the orthogonal complement of $V_0$ is given by $V_0^{\perp} = \Ker Q^{\frac12}$ since $Q^{\frac12}$ is a real symmetric semidefinite matrix. It follows that for all $\xi\in\mathbb R^n$,
\begin{equation}\label{04092018E1}
	\vert Q^{\frac12}\xi\vert^{2s} = \vert Q^{\frac12}\mathbb P_0\xi\vert^{2s}\le \Vert Q^{\frac12}\Vert^{2s}\langle\mathbb P_0\xi\rangle^{2s}.
\end{equation}
It follows from \eqref{04092018E1}, Theorem \ref{28092018T1} and the Plancherel theorem that there exists a positive constant $c>0$ such that for all $u\in D(\mathcal P)$,
\begin{multline}\label{04092018E2}
	\Vert \Tr^s(-Q\nabla^2_x)u\Vert_{L^2(\mathbb{R}^n)}
	\le \Vert Q^{\frac12}\Vert^{2s}\Vert\langle\mathbb P_0D_x\rangle^{2s}u\Vert_{L^2(\mathbb{R}^n)} \\
	\le c\Vert Q^{\frac12}\Vert^{2s}\left[\Vert \mathcal{P}u\Vert_{L^2(\mathbb{R}^n)} + \Vert u\Vert_{L^2(\mathbb{R}^n)}\right],
\end{multline}
since $\Tr^s(-Q\nabla^2_x) = \vert Q^{\frac{1}{2}}D_x\vert^{2s}$. Then, we deduce from \eqref{06022018E3} and \eqref{04092018E2} that for all $u\in D(\mathcal{P})$,
\begin{multline*}
	\Vert\langle Bx,\nabla_x\rangle u\Vert_{L^2(\mathbb{R}^n)} = \Vert \mathcal{P}u\Vert_{L^2(\mathbb{R}^n)} + \frac{1}{2}\Vert \Tr^s(-Q\nabla^2_x)u\Vert_{L^2(\mathbb{R}^n)} \\
	\le \big(1+\frac12c\Vert Q^{\frac12}\Vert^{2s}\big)\left[\Vert \mathcal{P}u\Vert_{L^2(\mathbb{R}^n)} + \Vert u\Vert_{L^2(\mathbb{R}^n)}\right].
\end{multline*}
This ends the proof of Corollary \ref{24052018C1}.

\section{Appendix}
\label{sec_appendix}

\subsection{About the Kalman rank condition} To begin this appendix, we prove the characterization of the Kalman rank condition we have used all over this work.

\begin{lem}\label{29082018E1} Let $B$ and $Q$ be real $n\times n$ matrices, with $Q$ symmetric positive semidefinite. The following assertions are equivalent : \\[5pt]
\textbf{1.} The Kalman rank condition \eqref{10052018E4} holds. \\[5pt]
\textbf{2.} There exists a non-negative integer $0\le r\le n-1$ satisfying
$$\Ker(Q^{\frac{1}{2}})\cap\Ker(Q^{\frac{1}{2}}B^T)\cap\ldots\cap\Ker(Q^{\frac{1}{2}}(B^T)^r) = \{0\}.$$
\end{lem}

\begin{proof} Using the notations of \eqref{10052018E4}, we have the following equivalences :
\begin{align*}
	&\ \Rank[B\ \vert\ Q^{\frac{1}{2}}] = n \\[5pt]
	\Leftrightarrow &\ \Ran[B\ \vert\ Q^{\frac{1}{2}}] = \mathbb R^n \\[5pt]
	\Leftrightarrow &\ \Ker([B\ \vert\ Q^{\frac{1}{2}}]^T) = (\Ran[B\ \vert\ Q^{\frac{1}{2}}])^{\perp} = \{0\} \\[5pt]
	\Leftrightarrow &\ \Ker(Q^{\frac{1}{2}})\cap\Ker(Q^{\frac{1}{2}}B^T)\cap\ldots\cap\Ker(Q^{\frac{1}{2}}(B^T)^{n-1}) = \Ker([B\ \vert\ Q^{\frac{1}{2}}]^T) = \{0\},
\end{align*}
where $\perp$ denotes the orthogonality with respect to the canonical Euclidean structure. This ends the proof of Lemma \ref{29082018E1}.
\end{proof}

\subsection{Convergence in Lebesgue spaces} In a second part, we recall the following classical measure theory result concerning the convergence in $L^p(\mathbb{R}^n)$. Its proof is given here for the convenience of the reader and for the sake of completeness of this work.

\begin{lem}\label{12022018L1} Let $p\in[1,+\infty)$. We consider $(f_k)_k$ a sequence of $L^p(\mathbb{R}^n)$ and $f\in L^p(\mathbb{R}^n)$ such that $(f_k)_k$ converges to $f$ almost everywhere in $\mathbb{R}^n$. Then : 
$$\lim_{k\rightarrow+\infty}\Vert f_k-f\Vert_{L^p(\mathbb{R}^n)} \Leftrightarrow \lim_{k\rightarrow+\infty}\Vert f_k\Vert_{L^p(\mathbb{R}^n)} = \Vert f\Vert_{L^p(\mathbb{R}^n)}.$$
\end{lem}

\begin{proof} We just need to prove the reciprocal implication. To that end, we consider the sequence $(g_k)_k$ of non-negative functions defined for all $k\geq0$ by
$$g_k = 2^{p-1}(\vert f_k\vert^p + \vert f\vert^p) - \vert f_k-f\vert^p\geq0.$$
Notice that the functions $g_k$ are non-negative since the following convexity inequality holds for all $x,y\in\mathbb{R}^n$,
$$(x+y)^p\le 2^{p-1}(x^p+y^p).$$
Since we have
$$\int_{\mathbb{R}^n}\liminf_{k\rightarrow\infty}g_k(x) dx = 2^p\int_{\mathbb{R}^n}\vert f(x)\vert^p dx,$$
and
$$\liminf_{k\rightarrow+\infty}\int_{\mathbb{R}^n}g_k(x) dx = 2^p\int_{\mathbb{R}^n}\vert f(x)\vert^pdx - \limsup_{k\rightarrow+\infty}\int_{\mathbb{R}^n}\vert f_k(x)-f(x)\vert^p dx,$$
it follows from the Fatou lemma that 
$$\limsup_{k\rightarrow+\infty}\int_{\mathbb{R}^n}\vert f_k(x)-f(x)\vert^p\ dx = 0.$$
This ends the proof of Lemma \ref{12022018L1}.
\end{proof}

\subsection{Some instrumental estimates} To end this appendix, we give the proof of four lemmas used in Section \ref{sec_frac_OUO} and Section \ref{sec_GreofOUs} : 

\begin{lem}\label{05022018L2} For all $r\geq1$, $q\in(0,+\infty)$ and $a_1,\ldots,a_r\in[0,+\infty)$, we have
\begin{gather}\label{17032018E1}
	(a_1+\ldots+a_r)^q\le r^{\left(q-1\right)_+}(a^q_1+\ldots+a^q_r),
\end{gather}
where
$$\left(q-1\right)_+ = \max(q-1,0).$$
\end{lem}

\begin{proof} If $a_j = 0$ for all $j\in\{1,\ldots,r\}$, the result is immediate. Therefore, we can assume that at least one of the $a_j$ is positive. \\[5pt]
\textbf{1. Case $0< q\le 1$ :} Since
$$\forall j\in\{1,\ldots,r\},\quad a_j\le a_1+\ldots+a_r,$$
we get that
$$\forall j\in\{1,\ldots,r\},\quad \frac{a_j}{a_1+\ldots+a_r}\le \left[\frac{a_j}{a_1+\ldots+a_r}\right]^q.$$
Then, \eqref{17032018E1} follows by summing up the previous inequalities for all $1\le j\le r$. \\[5pt]
\textbf{2. Case $q>1$ :} In this case, the convexity property of the function $t\mapsto t^q$ on $[0,+\infty)$ implies that
$$\frac{1}{r^q}(a_1+\ldots+a_r)^q\le\frac{1}{r}(a^q_1+\ldots+a^q_r).$$
This ends the proof of Lemma \ref{05022018L2}.
\end{proof}

\begin{lem}\label{2052018L1} For all $\xi,\eta\in\mathbb{R}^n$ and $q>0$, we have
$$2^{-(q-1)_+}\vert\xi\vert^q - \vert\eta\vert^q\le\vert\xi-\eta\vert^q,$$
where
$$\left(q-1\right)_+ = \max(q-1,0).$$
\end{lem}

\begin{proof} We deduce from Lemma \ref{05022018L2} that for all $\xi,\eta\in\mathbb R^n$ and $q>0$,
$$\vert\xi\vert^q = \vert\xi-\eta+\eta\vert^q\le 2^{(q-1)_+}(\vert\xi-\eta\vert^q+\vert\eta\vert^q).$$
This ends the proof of Lemma \ref{2052018L1}.
\end{proof}

\begin{lem}\label{16032018L1} For all $q>0$ and $\xi,\eta\in\mathbb{R}^n$,
$$\vert\vert\xi\vert^{q}-\vert\eta\vert^{q}\vert\le\left\{
\begin{array}{cl}
	q2^{(q-2)_+}\left(\vert\xi-\eta\vert^{q}+\min(\vert\xi\vert^{q-1},\vert\eta\vert^{q-1})\vert\xi-\eta\vert\right) & \text{when $q>1$}, \\[5pt]
	\vert\xi-\eta\vert^{q} & \text{when $0<q\le1$}.
\end{array}\right.$$
\end{lem}

\begin{proof} Let $q>0$ and $\xi,\eta\in\mathbb{R}^n$. We first assume that $0<q\le 1$. It follows from Lemma \ref{05022018L2} that
$\vert\xi+\eta\vert^{q}\le\vert\xi\vert^{q} + \vert\eta\vert^{q}$ and then, by a natural change of coordinate, this implies $\vert\vert\xi\vert^{q}-\vert\eta\vert^{q}\vert\le\vert\xi-\eta\vert^{q}$.
When $q>1$, we deduce from the differentiability of the function $\vert\cdot\vert^q$ the following equality
$$\vert\xi\vert^{q}-\vert\eta\vert^{q} = \int_0^1\frac{d}{dt}\vert\eta+t(\xi-\eta)\vert^{q}\ dt
= q\int_0^1\vert\eta+t(\xi-\eta)\vert^{q-2}(\eta+t(\xi-\eta))\cdot(\xi-\eta)\ dt.$$
Then, the Cauchy-Schwarz inequality and Lemma \ref{05022018L2} imply that
\begin{align*}
	\left\vert\vert\xi\vert^{q}-\vert\eta\vert^{q}\right\vert & \le q \vert\xi-\eta\vert\int_0^1\vert\eta+t(\xi-\eta)\vert^{q-1}\ dt \\
	& \le q2^{(q-2)_+}\ \vert\xi-\eta\vert\int_0^1\left(\vert\eta\vert^{q-1}+t^{q-1}\vert\xi-\eta\vert^{q-1}\right)\ dt \\
	& = q2^{(q-2)_+}\left(\vert\xi-\eta\vert^{q} + \vert\eta\vert^{q-1}\vert\xi-\eta\vert\right).
\end{align*}
Since $\xi$ et $\eta$ play symmetric roles, the proof of Lemma \ref{16032018L1} is ended.
\end{proof}

\end{document}